\documentclass[aap]{imsart}

\RequirePackage{amsthm,amsmath,amsfonts,amssymb, amstext}
\RequirePackage[numbers]{natbib}
\RequirePackage[colorlinks,citecolor=blue,urlcolor=blue]{hyperref}
\RequirePackage{graphicx}
\usepackage{bbm}
\usepackage{times}
\usepackage{tikz}
\usepackage{pgfplots}

\usepackage[T2A,T1]{fontenc}
\startlocaldefs

\newtheorem{theorem}{Theorem}[section]
\newtheorem{lemma}[theorem]{Lemma}
\theoremstyle{remark}
\newtheorem{remark}[theorem]{Remark}
\newtheorem{corollary}[theorem]{Corollary}
\newtheorem{proposition}[theorem]{Proposition}
\newtheorem{property}[theorem]{Property}
\newtheorem{definition}[theorem]{Definition}

\newtheorem{assumption}{Assumption}

\newcommand{\un}{\mathbbm{1}}

\DeclareMathOperator{\tr}{Tr}

\newcommand{\res}{Q}
\DeclareMathOperator{\diag}{Diag}

\endlocaldefs

\begin{document}

\begin{frontmatter}
\title{A Concentration of Measure and Random Matrix Approach\\ to Large Dimensional Robust Statistics}
\runtitle{Concentration of Measure \& Robust Statistics}

\begin{aug}


\author[B,A]{\fnms{Cosme} \snm{Louart}\ead[label=e1]{cosmelouart@gmail.com}}
\and
\author[B]{\fnms{Romain} \snm{Couillet}\ead[label=e2,mark]{romain.couillet@gipsa-lab.grenoble-inp.fr}}

\address[B]{GIPSA-lab, Universit\'e Grenoble-Alpes, France,
\printead{e2}}

\address[A]{CEA, LIST, site Nano-innov, France
\printead{e1}}
\end{aug}

\def\bibsection{\section*{References}}

\maketitle

\begin{abstract}%
This article studies the \emph{robust covariance matrix estimation} of a data collection $X = (x_1,\ldots,x_n)$ with $x_i = \sqrt \tau_i z_i + m$, where $z_i \in \mathbb R^p$ is a \textit{concentrated vector} (e.g., an elliptical random vector), $m\in \mathbb R^p$ a deterministic signal and $\tau_i\in \mathbb R$ a scalar perturbation of possibly large amplitude, under the assumption where both $n$ and $p$ are large. This estimator is defined as the fixed point of a function which we show is contracting for a so-called \textit{stable semi-metric}. We exploit this semi-metric along with concentration of measure arguments to prove the existence and uniqueness of the robust estimator as well as evaluate its limiting spectral distribution.
\end{abstract}

\end{frontmatter}

\section{Introduction}

Robust estimators of covariance (or scatter) are necessary ersatz for the classical sample covariance matrix when the dataset $X = (x_1,\ldots,x_n)$ present some diverging statistical properties, such as unbounded second moments of the $x_i$'s. We study here the M-estimator of scatter $\hat C$ initially introduced in \citep{HUB64} defined as the solution (if it exists) to the following fixed point equation:
\begin{align}
\label{eq:hatC}
  \hat C = \frac{1}{n} \sum_{i=1}^n u \left(\frac{1}{n}x_i^T(\hat C +\gamma I_p)^{-1}x_i\right) x_ix_i^T,
\end{align}
where $\gamma>0$ is a regularization parameter and $u: \mathbb R^+ \to \mathbb R^+$ a mapping that tends to zero at $+\infty$, and whose object is to control outlying data. The literature in this domain has so far divided the study of $\hat C$ into (i) a first exploration of conditions for its existence and uniqueness as a \emph{deterministic} solution to \eqref{eq:hatC} (e.g., \citep{HUB64,MAR76,TYL87}) and (ii) an independent analysis of its statistical properties when seen as a random object (in the large $n$ regime \citep{CHI08} or in the large $n,p$ regime \citep{COU14,ZHA14}).

\medskip

In the present article, we claim that the study of the conditions of existence (i) and statistical behavior (ii) of $\hat C$ can be conveniently carried out jointly. Specifically, by means of a flexible framework based on concentration of measure theory and on a new stable semi-metric argument, we simultaneously explore the existence and large dimensional ($n,p$ large) spectral properties of $\hat C$. Our findings may be summarized as the following three main contributions to robust statistics and more generally to large dimensional statistics.

First, the proposed concentration of measure framework has the advantage of relaxing the assumptions of independence in the entries of $x_i$ made in previous works \citep{COU14,ZHA14}, thereby allowing for possibly complex and quite realistic data models.
In detail, our data model decomposes $x_i$ as $x_i = \sqrt{\tau_i} z_i + m$ where the $z_1,\ldots, z_n$ are independent random vectors satisfying a concentration of measure hypothesis (in particular, the $z_i$'s could arise from a very generic generative model, .e.g, $z_i=h(\tilde z_i)$ for $\tilde z_i\sim\mathcal N(0,I_q)$ and $h:\mathbb{R}^q\to\mathbb{R}^p$ a $1$-Lipschitz mapping), $m$ is a deterministic vector (a signal or information common to all data) and $\tau_i$ are arbitrary (possibly large) deterministic values.\footnote{We may alternatively assume the $\tau_i$ random independent of $Z = (z_1,\ldots,z_n)$.} This setting naturally arises in many engineering applications, such as in antenna array processing (radar, brain signal processing, etc.) where the $\tau_i$'s model noise impulsiveness and $m$ is an informative signal to be detected by the experimenter \citep{OVA11}, or in statistical finance where the $x_i$'s model asset returns with high volatility and $m$ is the market leading direction \citep{YAN14}. Besides, the hypothesis made on $z_i$ is adapted to the generative modelling of possibly extremely complex data: it in particular encompasses all data models produced by generative neural networks, such as the now popular GANs (generative adversarial neural networks \citep{GOO14}).

Second, as compared to previous works in the field \citep{MAR76,KEN91,ollila2014regularized,OLI20,COU14}, our frameworks allows for the relaxation of some of the classically posed constraints on the mapping $u$ made. Specifically, $u$ is here only required to be $1$-Lipschitz with respect to the ``stable semi-metric'' (defined in the course of the article), which is equivalent to assuming that $t\mapsto tu(t)$ is non-decreasing and that $t\mapsto u(t)/t$ is non-increasing. The semi-metric naturally arises when studying the \emph{resolvent} $(\hat C +\gamma I_p)^{-1}$ of $\hat C$, which is at the core of our large $p,n$ analysis of $\hat C$, using modern tools from random matrix theory. To establish concentration properties in the large dimensional regime on $\hat C$, under our framework, the function $u$ is nonetheless further requested to be such that $t\to tu(t)$ is strictly smaller than $1$ ($\hat C$ is however still defined without this condition). Yet, and most importantly, $u$ needs not be a non-increasing function, as demanded by most works in the field.

Third, the ``Lipschitz and stable semi-metric'' properties of the model are consistently articulated so as to propagate the concentration properties from $Z$ to the robust scatter matrix $\hat C$. The core technical result allowing for this articulation is Theorem~\ref{the:concentration_point_fixe_stable}. This combined framework provides the rate of convergence of the Stieltjes transform of the spectral distribution of $\hat C$ to its large $n,p$ limit along with conditions guaranteeing the possibility to recover the signal $m$ from the asymptotic statistical properties of $\hat C$.

\section{Main Result}\label{sec:main_result}
Let us note, for $k \in \mathbb N$, $[k] \equiv \{1,\ldots, k\}$; $\mathbb R^+ \equiv \{x \in \mathbb R, x>0\}$; for any $A \subset \mathbb{C}$,  $\mathcal M_{p,n}(A)$, the set of real matrices of size $p \times n$ having value in $A$, that we endow with the spectral norm $\left\Vert M\right\Vert = \sup \{\left\Vert Mu\right\Vert, u \in \mathbb C^n, \Vert u \Vert \leq 1\}$, for $M \in \mathcal M_{p,n}$, the Frobenius norm $\left\Vert M\right\Vert_F = \sqrt{\sum_{\genfrac{}{}{0pt}{2}{1\leq i\leq p}{1\leq j\leq n}} |M_{i,j}|^2}$ and the nuclear norm $\|M\|_* = \tr((MM^T)^{1/2})$. 
We further note $\mathcal D_n(A) \equiv \{\Delta \in \mathcal M_n(A) \ |\ i\neq j \Leftrightarrow \Delta_{i,j} = 0\}$, the set of diagonal matrices having value in $A$ (it is endowed with the spectral notma on $\mathcal M_n(\mathbb{C})\equiv \mathcal M_{n,n}(\mathbb{C})$). Given $\Delta \in \mathcal D_n$, we let $\Delta_1,\ldots, \Delta_n \in \mathbb R$, be its diagonal elements, $\Delta = \diag(\Delta_i)_{1\leq i \leq n}$ so that $\left\Vert \Delta\right\Vert = \sup \{\left\vert \Delta_i\right\vert, i \in [n]\}$ (where $[n] = \{1,\ldots, n\}$); we define then $\mathcal D_n^+ \equiv \mathcal D_n(\mathbb{R}^+)$.

We place ourselves under the random matrix regime where $p$, the size of data $x_1,\ldots, x_n \in \mathbb R^p$ is of the same order as $n$, the number of data -- for practical use, imagine that $10^{-2}\leq \frac{p}{n} \leq 10^2$. The convergence results will be expressed as functions of the quasi asymptotic quantities $p$ and $n$ that are thought of as tending to infinity (in practice our results are extremely accurate already for $p,n \geq 100$). We will then work with the notations $a_{n,p}\leq  O(b_{n,p})$ or $a_{n,p}\geq  O(b_{n,p})$ to signify that there exists a constant $K$ independent of $p$ and $n$ such that $a_{n,p}\leq K b_{n,p}$ or $a_{n,p}\geq K b_{n,p}$, respectively, and to simplify the notation, most of the time, the indices $n,p$ will be omitted. In particular we have $O(n) \leq p\leq O(n)$.
Our hypotheses concern four central objects:
\begin{itemize}
  \item $Z=(z_1,\ldots z_n)\in \mathcal M_{p,n}$  satisfies the concentration of measure phenomenon (to be presented later);
  all the random vectors $z_1,\ldots, z_n$ are independent and $\sup_{1\leq i\leq n} \|\mathbb E[z_i]\|\leq O(1)$;
  \item $\tau = \diag(\tau_1,\ldots,\tau_n)\in \mathcal D_n^+$ satisfy $\forall i\in[n], \tau_i>0$ and $\frac{1}{n}\sum_{i=1}^n \tau_i \leq O(1)$;
  \item $m \in \mathbb R^p$ and $\|m\|\leq O(1)$;
  \item $u: \mathbb R^+ \to \mathbb R^+$ is bounded, 
  $t \mapsto tu(t)$ is non-decreasing, $t \mapsto \frac{u(t)}{t}$ is non-increasing and $\forall t>0$: $tu(t) < 1$.
\end{itemize}

\begin{figure}\label{fig:exemple_u}
\centering
\begin{tabular}{ccc}

\begin{tikzpicture}   
  \begin{axis}[ y label style={at={(-0.12,0.5)}}, width = 0.41\textwidth, legend cell align = left, legend style={draw=none},ytick ={0,1}, ymax = 1.5, xtick={0,10}, xmax = 12] 
    \addplot[sharp plot,red!60!black, line width = 1,domain=1:40,samples=100,line legend, smooth] coordinates {(0.0, 0.0)(0.1, 0.3162)(0.2, 0.4472)(0.3, 0.5477)(0.4, 0.6325)(0.5, 0.7071)(0.6, 0.7746)(0.7, 0.8367)(0.8, 0.8944)(0.9, 0.9487)(1.0, 0.9)(1.1, 0.8182)(1.2, 0.75)(1.3, 0.6923)(1.4, 0.6429)(1.5, 0.6)(1.6, 0.5625)(1.7, 0.5294)(1.8, 0.5)(1.9, 0.4737)(2.0, 0.45)(2.1, 0.4286)(2.2, 0.4091)(2.3, 0.3913)(2.4, 0.375)(2.5, 0.36)(2.6, 0.3462)(2.7, 0.3333)(2.8, 0.3214)(2.9, 0.3103)(3.0, 0.3)(3.1, 0.2903)(3.2, 0.2812)(3.3, 0.2727)(3.4, 0.2647)(3.5, 0.2571)(3.6, 0.25)(3.7, 0.2432)(3.8, 0.2368)(3.9, 0.2308)(4.0, 0.225)(4.1, 0.2195)(4.2, 0.2143)(4.3, 0.2093)(4.4, 0.2045)(4.5, 0.2)(4.6, 0.1957)(4.7, 0.1915)(4.8, 0.1875)(4.9, 0.1837)(5.0, 0.18)(5.1, 0.1765)(5.2, 0.1731)(5.3, 0.1698)(5.4, 0.1667)(5.5, 0.1636)(5.6, 0.1607)(5.7, 0.1579)(5.8, 0.1552)(5.9, 0.1525)(6.0, 0.15)(6.1, 0.1475)(6.2, 0.1452)(6.3, 0.1429)(6.4, 0.1406)(6.5, 0.1385)(6.6, 0.1364)(6.7, 0.1343)(6.8, 0.1324)(6.9, 0.1304)(7.0, 0.1286)(7.1, 0.1268)(7.2, 0.125)(7.3, 0.1233)(7.4, 0.1216)(7.5, 0.12)(7.6, 0.1184)(7.7, 0.1169)(7.8, 0.1154)(7.9, 0.1139)(8.0, 0.1125)(8.1, 0.1111)(8.2, 0.1098)(8.3, 0.1084)(8.4, 0.1071)(8.5, 0.1059)(8.6, 0.1047)(8.7, 0.1034)(8.8, 0.1023)(8.9, 0.1011)(9.0, 0.1)(9.1, 0.0989)(9.2, 0.0978)(9.3, 0.0968)(9.4, 0.0957)(9.5, 0.0947)(9.6, 0.0938)(9.7, 0.0928)(9.8, 0.0918)(9.9, 0.0909)(10.0, 0.09)(10.1, 0.0891)(10.2, 0.0882)(10.3, 0.0874)(10.4, 0.0865)(10.5, 0.0857)(10.6, 0.0849)(10.7, 0.0841)(10.8, 0.0833)(10.9, 0.0826)(11.0, 0.0818)(11.1, 0.0811)(11.2, 0.0804)(11.3, 0.0796)(11.4, 0.0789)(11.5, 0.0783)(11.6, 0.0776)(11.7, 0.0769)(11.8, 0.0763)(11.9, 0.0756)(12.0, 0.075)};
    \addplot[sharp plot,blue!60!black,, line width = 1,domain=1:40,samples=100,line legend, smooth] coordinates {(0.0, 0.0)(0.1, 0.097)(0.2, 0.1842)(0.3, 0.2602)(0.4, 0.3257)(0.5, 0.382)(0.6, 0.4306)(0.7, 0.4726)(0.8, 0.5092)(0.9, 0.5414)(1.0, 0.5697)(1.1, 0.5949)(1.2, 0.6174)(1.3, 0.6376)(1.4, 0.6559)(1.5, 0.6724)(1.6, 0.6874)(1.7, 0.7012)(1.8, 0.7138)(1.9, 0.7254)(2.0, 0.7361)(2.1, 0.746)(2.2, 0.7552)(2.3, 0.7638)(2.4, 0.7718)(2.5, 0.7793)(2.6, 0.7864)(2.7, 0.7933)(2.8, 0.7994)(2.9, 0.8049)(3.0, 0.8105)(3.1, 0.8157)(3.2, 0.8205)(3.3, 0.8255)(3.4, 0.83)(3.5, 0.8342)(3.6, 0.8383)(3.7, 0.8418)(3.8, 0.8454)(3.9, 0.8489)(4.0, 0.8522)(4.1, 0.8554)(4.2, 0.8584)(4.3, 0.8613)(4.4, 0.8641)(4.5, 0.8668)(4.6, 0.8693)(4.7, 0.8722)(4.8, 0.8746)(4.9, 0.877)(5.0, 0.8792)(5.1, 0.8814)(5.2, 0.8834)(5.3, 0.8855)(5.4, 0.8874)(5.5, 0.8893)(5.6, 0.8911)(5.7, 0.8929)(5.8, 0.8946)(5.9, 0.8962)(6.0, 0.8977)(6.1, 0.8989)(6.2, 0.9005)(6.3, 0.9023)(6.4, 0.9035)(6.5, 0.9047)(6.6, 0.9064)(6.7, 0.9078)(6.8, 0.9087)(6.9, 0.9098)(7.0, 0.911)(7.1, 0.9121)(7.2, 0.9132)(7.3, 0.9144)(7.4, 0.9159)(7.5, 0.9166)(7.6, 0.9175)(7.7, 0.919)(7.8, 0.9196)(7.9, 0.9205)(8.0, 0.9215)(8.1, 0.9223)(8.2, 0.9237)(8.3, 0.9245)(8.4, 0.9249)(8.5, 0.9262)(8.6, 0.9266)(8.7, 0.9276)(8.8, 0.9286)(8.9, 0.9289)(9.0, 0.9297)(9.1, 0.9307)(9.2, 0.9314)(9.3, 0.9389)(9.4, 0.9485)(9.5, 0.9589)(9.6, 0.9689)(9.7, 0.9783)(9.8, 0.9885)(9.9, 0.9981)(10.0, 1.0)(10.1, 1.001)(10.2, 1.0113)(10.3, 1.0213)(10.4, 1.0315)(10.5, 1.0413)(10.6, 1.0515)(10.7, 1.061)(10.8, 1.0716)(10.9, 1.0814)(11.0, 1.0913)(11.1, 1.1013)(11.2, 1.1114)(11.3, 1.1215)(11.4, 1.1317)(11.5, 1.141)(11.6, 1.1513)(11.7, 1.1616)(11.8, 1.1711)(11.9, 1.1814)(12.0, 1.1918)};
  \end{axis} 
\end{tikzpicture}
\begin{tikzpicture}   
  \begin{axis}[ y label style={at={(-0.08,0.5)}}, width = 0.41\textwidth, legend cell align = left, legend style={draw=none},ytick ={0,1}, ymax = 1.5, xtick={0,10}, xmax = 12] 
    \addplot[red!60!black, line width = 1,domain=1:40,samples=100, smooth] coordinates {(0.001, 1.0)(0.0015, 0.75)(0.002, 1.0)(0.0029, 0.75)(0.0039, 1.0)(0.0059, 0.75)(0.0078, 1.0)(0.0117, 0.75)(0.0156, 1.0)(0.0234, 0.75)(0.0312, 1.0)(0.0469, 0.75)(0.0625, 1.0)(0.0938, 0.75)(0.125, 1.0)(0.1875, 0.75)(0.25, 1.0)(0.375, 0.75)(0.5, 1.0)(0.75, 0.75)(1.0, 0.9)(1.1, 0.8182)(1.2, 0.75)(1.3, 0.6923)(1.4, 0.6429)(1.5, 0.6)(1.6, 0.5625)(1.7, 0.5294)(1.8, 0.5)(1.9, 0.4737)(2.0, 0.45)(2.1, 0.4286)(2.2, 0.4091)(2.3, 0.3913)(2.4, 0.375)(2.5, 0.36)(2.6, 0.3462)(2.7, 0.3333)(2.8, 0.3214)(2.9, 0.3103)(3.0, 0.3)(3.1, 0.2903)(3.2, 0.2812)(3.3, 0.2727)(3.4, 0.2647)(3.5, 0.2571)(3.6, 0.25)(3.7, 0.2432)(3.8, 0.2368)(3.9, 0.2308)(4.0, 0.225)(4.1, 0.2195)(4.2, 0.2143)(4.3, 0.2093)(4.4, 0.2045)(4.5, 0.2)(4.6, 0.1957)(4.7, 0.1915)(4.8, 0.1875)(4.9, 0.1837)(5.0, 0.18)(5.1, 0.1765)(5.2, 0.1731)(5.3, 0.1698)(5.4, 0.1667)(5.5, 0.1636)(5.6, 0.1607)(5.7, 0.1579)(5.8, 0.1552)(5.9, 0.1525)(6.0, 0.15)(6.1, 0.1475)(6.2, 0.1452)(6.3, 0.1429)(6.4, 0.1406)(6.5, 0.1385)(6.6, 0.1364)(6.7, 0.1343)(6.8, 0.1324)(6.9, 0.1304)(7.0, 0.1286)(7.1, 0.1268)(7.2, 0.125)(7.3, 0.1233)(7.4, 0.1216)(7.5, 0.12)(7.6, 0.1184)(7.7, 0.1169)(7.8, 0.1154)(7.9, 0.1139)(8.0, 0.1125)(8.1, 0.1111)(8.2, 0.1098)(8.3, 0.1084)(8.4, 0.1071)(8.5, 0.1059)(8.6, 0.1047)(8.7, 0.1034)(8.8, 0.1023)(8.9, 0.1011)(9.0, 0.1)(9.1, 0.0989)(9.2, 0.0978)(9.3, 0.0968)(9.4, 0.0957)(9.5, 0.0947)(9.6, 0.0938)(9.7, 0.0928)(9.8, 0.0918)(9.9, 0.0909)(10.0, 0.09)(10.1, 0.0891)(10.2, 0.0882)(10.3, 0.0874)(10.4, 0.0865)(10.5, 0.0857)(10.6, 0.0849)(10.7, 0.0841)(10.8, 0.0833)(10.9, 0.0826)(11.0, 0.0818)(11.1, 0.0811)(11.2, 0.0804)(11.3, 0.0796)(11.4, 0.0789)(11.5, 0.0783)(11.6, 0.0776)(11.7, 0.0769)(11.8, 0.0763)(11.9, 0.0756)(12.0, 0.075)};
    \addplot[sharp plot,blue!60!black,, line width = 1,domain=1:40,samples=100,line legend, smooth] coordinates {(0.0001, 0.0001)(0.0101, 0.01)(0.0201, 0.0198)(0.0301, 0.0293)(0.0401, 0.0387)(0.0501, 0.0482)(0.0601, 0.057)(0.0701, 0.0656)(0.0801, 0.0747)(0.0901, 0.0838)(0.1001, 0.0931)(0.1101, 0.1012)(0.1201, 0.1088)(0.1301, 0.116)(0.1401, 0.1231)(0.1501, 0.1309)(0.1601, 0.1392)(0.1701, 0.1473)(0.1801, 0.1555)(0.1901, 0.1638)(0.2001, 0.1721)(0.2101, 0.1805)(0.2201, 0.1888)(0.2301, 0.195)(0.2401, 0.2011)(0.2501, 0.2071)(0.2601, 0.2129)(0.2701, 0.2185)(0.2801, 0.224)(0.2901, 0.2292)(0.3, 0.2343)(0.4, 0.2929)(0.5, 0.3594)(0.6, 0.4039)(0.7, 0.4349)(0.8, 0.4607)(0.9, 0.4821)(1.0, 0.5001)(1.1, 0.5332)(1.2, 0.5661)(1.3, 0.5987)(1.4, 0.6317)(1.5, 0.6662)(1.6, 0.7028)(1.7, 0.7428)(1.8, 0.7604)(1.9, 0.7706)(2.0, 0.781)(2.1, 0.7895)(2.2, 0.7985)(2.3, 0.8056)(2.4, 0.8128)(2.5, 0.8201)(2.6, 0.8264)(2.7, 0.8322)(2.8, 0.837)(2.9, 0.842)(3.0, 0.8468)(4.0, 0.8824)(4.1, 0.8851)(4.2, 0.8884)(4.3, 0.8908)(4.4, 0.8924)(4.5, 0.8946)(4.6, 0.8968)(4.7, 0.8989)(4.8, 0.9017)(4.9, 0.9027)(5.0, 0.9054)(5.1, 0.9063)(5.2, 0.9081)(5.3, 0.9105)(5.4, 0.9113)(5.5, 0.9136)(5.6, 0.9143)(5.7, 0.9165)(5.8, 0.9171)(5.9, 0.9192)(6.0, 0.9206)(6.1, 0.921)(6.2, 0.923)(6.3, 0.9242)(6.4, 0.9245)(6.5, 0.9264)(6.6, 0.9267)(6.7, 0.9286)(6.8, 0.9287)(6.9, 0.9306)(7.0, 0.9315)(7.1, 0.9316)(7.2, 0.9325)(7.3, 0.9334)(7.4, 0.9351)(7.5, 0.936)(7.6, 0.9359)(7.7, 0.9367)(7.8, 0.9375)(7.9, 0.9391)(8.0, 0.9399)(8.1, 0.9401)(8.2, 0.9413)(8.3, 0.9411)(8.4, 0.9418)(8.5, 0.943)(8.6, 0.944)(8.7, 0.9437)(8.8, 0.9444)(8.9, 0.9455)(9.0, 0.9464)(9.1, 0.9465)(9.2, 0.9467)(9.3, 0.9472)(9.4, 0.9485)(9.5, 0.9589)(9.6, 0.9689)(9.7, 0.9783)(9.8, 0.9885)(9.9, 0.9981)(10.0, 1.0)(10.1, 1.001)(10.2, 1.0113)(10.3, 1.0213)(10.4, 1.0315)(10.5, 1.0413)(10.6, 1.0515)(10.7, 1.061)(10.8, 1.0716)(10.9, 1.0814)(11.0, 1.0913)(11.1, 1.1013)(11.2, 1.1114)(11.3, 1.1215)(11.4, 1.1317)(11.5, 1.141)(11.6, 1.1513)(11.7, 1.1616)(11.8, 1.1711)(11.9, 1.1814)(12.0, 1.1918)};
  \end{axis} 
\end{tikzpicture}
\begin{tikzpicture}   
  \begin{axis}[ y label style={at={(-0.12,0.5)}}, width = 0.41\textwidth, legend cell align = left, legend style={draw=none},ytick ={0,1}, ymax = 1.5, xtick={0,10}, xmax = 12] 
    \addplot[sharp plot,red!60!black, line width = 1,domain=1:40,samples=100,line legend, smooth] coordinates {(0.0, Inf)(0.1, 3.1623)(0.2, 2.2361)(0.3, 1.8257)(0.4, 1.5811)(0.5, 1.4142)(0.6, 1.291)(0.7, 1.1952)(0.8, 1.118)(0.9, 1.0541)(1.0, 1.0)(1.1, 0.9535)(1.2, 
0.9129)(1.3, 0.8771)(1.4, 0.8452)(1.5, 0.8165)(1.6, 0.7906)(1.7, 0.767)(1.8, 0.7454)(1.9, 0.7255)(2.0, 0.7071)(2.1, 0.6901)(2.2, 0.6742)(2.3, 0.6594)(2.4, 0.6455)(2.5, 0.6325)(2.6, 0.6202)(2.7, 0.6086)(2.8, 0.5976)(2.9, 0.5872)(3.0, 0.5774)(3.1, 0.568)(3.2, 0.559)(3.3, 0.5505)(3.4, 0.5423)(3.5, 0.5345)(3.6, 0.527)(3.7, 0.5199)(3.8, 0.513)(3.9, 0.5064)(4.0, 0.5)(4.1, 0.4939)(4.2, 0.488)(4.3, 0.4822)(4.4, 0.4767)(4.5, 0.4714)(4.6, 0.4663)(4.7, 0.4613)(4.8, 0.4564)(4.9, 0.4518)(5.0, 0.4472)(5.1, 0.4428)(5.2, 0.4385)(5.3, 0.4344)(5.4, 0.4303)(5.5, 0.4264)(5.6, 0.4226)(5.7, 0.4189)(5.8, 0.4152)(5.9, 0.4117)(6.0, 0.4082)(6.1, 0.4049)(6.2, 0.4016)(6.3, 0.3984)(6.4, 0.3953)(6.5, 0.3922)(6.6, 0.3892)(6.7, 0.3863)(6.8, 0.3835)(6.9, 0.3807)(7.0, 0.378)(7.1, 0.3753)(7.2, 
0.3727)(7.3, 0.3701)(7.4, 0.3676)(7.5, 0.3651)(7.6, 0.3627)(7.7, 0.3604)(7.8, 0.3581)(7.9, 0.3558)(8.0, 0.3536)(8.1, 0.3514)(8.2, 0.3492)(8.3, 0.3471)(8.4, 0.345)(8.5, 0.343)(8.6, 0.341)(8.7, 0.339)(8.8, 0.3371)(8.9, 0.3352)(9.0, 0.3333)(9.1, 0.3315)(9.2, 0.3297)(9.3, 0.3279)(9.4, 0.3262)(9.5, 0.3244)(9.6, 
0.3227)(9.7, 0.3211)(9.8, 0.3194)(9.9, 0.3178)(10.0, 0.3162)(10.1, 0.3147)(10.2, 0.3131)(10.3, 0.3116)(10.4, 0.3101)(10.5, 0.3086)(10.6, 0.3071)(10.7, 0.3057)(10.8, 0.3043)(10.9, 0.3029)(11.0, 0.3015)(11.1, 0.3002)(11.2, 0.2988)(11.3, 0.2975)(11.4, 0.2962)(11.5, 0.2949)(11.6, 0.2936)(11.7, 0.2924)(11.8, 0.2911)(11.9, 0.2899)(12.0, 0.2887)};
    \addplot[sharp plot,blue!60!black,, line width = 1,domain=1:40,samples=100,line legend, smooth] coordinates {(0.0, 0.0)(0.1, 0.073)(0.2, 0.1285)(0.3, 0.1747)(0.4, 0.2148)(0.5, 0.2503)(0.6, 0.2817)(0.7, 0.3103)(0.8, 0.3363)(0.9, 0.3602)(1.0, 0.3822)(1.1, 0.4026)(1.2, 0.4216)(1.3, 0.4393)(1.4, 0.4558)(1.5, 0.471)(1.6, 0.4855)(1.7, 0.4993)(1.8, 0.5122)(1.9, 0.5245)(2.0, 0.5362)(2.1, 0.5472)(2.2, 0.5577)(2.3, 0.5678)(2.4, 0.5773)(2.5, 0.5864)(2.6, 0.5952)(2.7, 0.6035)(2.8, 0.6115)(2.9, 0.6192)(3.0, 0.6266)(3.1, 0.6336)(3.2, 0.6405)(3.3, 0.647)(3.4, 0.6533)(3.5, 0.6594)(3.6, 0.6653)(3.7, 0.6709)(3.8, 0.6764)(3.9, 0.6817)(4.0, 0.6868)(4.1, 0.6918)(4.2, 0.6966)(4.3, 0.7012)(4.4, 0.7057)(4.5, 0.7101)(4.6, 0.7143)(4.7, 0.7184)(4.8, 0.7224)(4.9, 0.7263)(5.0, 0.7301)(5.1, 0.7337)(5.2, 0.7373)(5.3, 0.7408)(5.4, 0.7442)(5.5, 0.7475)(5.6, 0.7507)(5.7, 0.7538)(5.8, 0.7568)(5.9, 0.7598)(6.0, 0.7627)(6.1, 0.7655)(6.2, 0.7683)(6.3, 0.771)(6.4, 0.7736)(6.5, 0.7762)(6.6, 0.7787)(6.7, 0.7812)(6.8, 0.7836)(6.9, 0.786)(7.0, 0.7883)(7.1, 0.7905)(7.2, 0.7927)(7.3, 0.7949)(7.4, 0.797)(7.5, 0.7991)(7.6, 0.8011)(7.7, 0.8031)(7.8, 0.805)(7.9, 0.8069)(8.0, 0.8088)(8.1, 0.8106)(8.2, 0.8124)(8.3, 0.8142)(8.4, 0.8159)(8.5, 0.8176)(8.6, 0.8193)(8.7, 0.8209)(8.8, 0.8226)(8.9, 0.8241)(9.0, 0.8257)(9.1, 0.8272)(9.2, 0.8287)(9.3, 0.8302)(9.4, 0.8316)(9.5, 0.8331)(9.6, 0.8345)(9.7, 0.8358)(9.8, 0.8372)(9.9, 0.8385)(10.0, 0.8398)(10.1, 0.8411)(10.2, 0.8424)(10.3, 0.8436)(10.4, 0.8449)(10.5, 0.8461)(10.6, 0.8472)(10.7, 0.8484)(10.8, 0.8496)(10.9, 0.8507)(11.0, 0.8518)(11.1, 0.8529)(11.2, 0.854)(11.3, 0.8551)(11.4, 0.8561)(11.5, 0.8572)(11.6, 0.8582)(11.7, 0.8592)(11.8, 0.8602)(11.9, 0.8612)(12.0, 0.8622)};
    \legend{$y = u(t)$, $y = \eta(t)$}
  \end{axis} 
\end{tikzpicture}
\end{tabular}
\caption{Three stable mappings $u$ and their associated $\eta$ mappings. \textbf{(Left)} $u: t \mapsto \max (\sqrt{t},\frac{1}{t})$, \textbf{(Center)} mapping introduced in Remark~\ref{rem:stable_non_continu_en_zero} having no limit in $0$, \textbf{(Right)} $u: t \mapsto \frac{1}{\sqrt{t}}$, and $\eta: t \mapsto \sqrt{\sqrt{\frac{x^2}{4}+1}-\frac{x}{2}}$ ; this last choice does not satisfy our hypotheses because $u$ is not bounded and $\lim_{t \to \infty}t u(t) = \infty \geq 1$ (then $t \mapsto \frac{\eta(t)}{t}$ is not bounded from below).}
\end{figure}
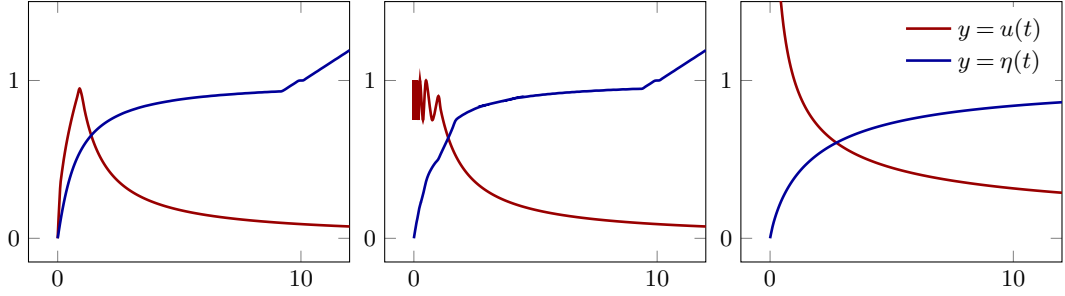

Those conditions are sufficient to retrieve part of the statistical properties of $Z$ and of the signal $m$ from the data matrix $$X = Z\sqrt \tau + m \un^T $$
through the robust scatter matrix $\hat C$ defined in Equation~\eqref{eq:hatC}. The standard sample covariance matrix $\frac{1}{n}XX^T$ instead inefficiently estimates some of these statistics due to the presence of possibly large (outlying) $\tau_i$'s (although $\frac{1}{n}\sum_{i=1}^n \tau_i \leq O(1)$, it is allowed for some $\tau_i$'s to be of order $\tau_i \geq O(n)$). The robust scatter matrix controls this outlying behavior by mitigating the impact of the high energy data $x_i$ with the tapering action of the mapping $u$ induced by the hypothesis $tu(t) <1$ (see Figure~\ref{fig:example}). 

Introducing the diagonal matrix $\hat \Delta$ solution to the fixed point equation:
\begin{align*}
   \hat \Delta = \frac{1}{n}x_i^T \left(\frac{1}{n}X^Tu(\hat \Delta) X + \gamma I_p\right)^{-1}x_i,
 \end{align*} 
(with $u(\cdot)$ operating entry-wise on the diagonal elements of $\hat\Delta$)
the robust scatter matrix is simply $\hat C = \frac{1}{n}Xu(\hat \Delta) X^T$, and the tapering action is revealed by low values of $u(\tilde \Delta)_{ii}$ when $\tau_i$ is large.
As shown on the central display of Figure~\ref{fig:example}, compared to $\frac{1}{n}XX^T$, $\hat C = \frac{1}{n}Xu(\hat \Delta) X^T$ has a cleaner spectral behavior which lets appear the signal induced by $m$ as an isolated eigenvalue-eigenvector pair. This eigenvector can then be exploited to estimate $m$ (this is a classical random matrix inference problem, which however is beyond the scope of the present article).

\medskip

This paper precisely shows that the spectral distribution of $\hat C$ is asymptotical equivalent to the spectral distribution of $\frac{1}{n}Z^T U Z$ where $U$ is a deterministic diagonal matrix satisfying $\|U\|\leq O(1)$. Interestingly, the definition of $U$ merely depends on the second order moments of $z_1,\ldots z_n$ which we denote, $\forall i \in [n]$, $C_i \equiv \mathbb E[z_iz_i^T]$, on the vector $\tau\in \mathbb R^n$ of the $\tau_i$'s, on the function $u$, but not on the signal $m$. The definition of $U$ relies on the introduction of a function $\eta:\mathbb R^+ \to \mathbb R^+$ derived from $u$ and defined as the solution to
\begin{align*}
   \forall t \in \mathbb R^+ : \ \ \eta(t) = \frac{t}{1+t u(\eta(t))}
 \end{align*} 
 and on the diagonal matrix $\Lambda_z:\mathcal D_n^+ \to \mathcal D_n^+$. For any $z\in \mathbb R^+$ and $\Delta\in \mathcal D_n^+$, $\Lambda_z(\Delta)$ is defined as the unique solution to the $n$ equations:
    \begin{align*}
      \forall i \in [n], \ \Lambda_z(\Delta)_i = \frac{1}{n} \tr \left(C_i \left(\frac{1}{n}\sum_{j=1}^n\frac{C_j \Delta_j}{1+ \Delta_j \Lambda_z(\Delta)_j} + z I_p\right)^{-1}\right).
    \end{align*}

Introducing the resolvent $Q_z \equiv (z I_p + \hat C )^{-1}$, 
we have the concentration:
\begin{theorem}\label{the:Estimation_transforme_stieltjes}
  For any $z\geq O(1)$, and any deterministic matrix $A \in \mathcal{M}_{p}$ such that $\|A\|_* \leq O(1)$ there exist two constants $C,c>0$ ($C,c \sim O(1)$) such that, for any $0<\varepsilon\leq 1$,
    \begin{align*}
      \mathbb P \left(\left\vert \tr(A Q_{z}) - \frac{1}{p}\tr \left(A \left( \frac{1}{n}\sum_{i=1}^n\frac{U_i C_i}{1+\Lambda_z(U )_iU_i}+z I_p \right)^{-1} \right)\right\vert \geq \varepsilon\right) \leq C e^{-cn\varepsilon^2}
    \end{align*}
    where $U={\rm diag}(U_1,\ldots,U_n)\in \mathcal D_n^+$ satisfies $U \leq O(1)$ and is the unique solution to the equation:
    \begin{align}\label{eq:equation_U_theorem}
      U = \tau \cdot u \circ \eta \left(\tau \Lambda_\gamma(U)\right)
    \end{align}
    (the mappings $u$ and $\eta$ are applied entry-wise on the diagonal terms on $\mathcal D_n^+$).
\end{theorem}
The proof of the existence and uniqueness of $U$ is based on contractiveness arguments (see Theorem~\ref{the:point_fixe_fonction_contractante_de_D_n_borne_superieurement} below), therefore, in practice, a precise estimate of $U$ is merely obtained from successive iteration of \eqref{eq:equation_U_theorem}.

Employing this theorem with $A = \frac{1}{p}I_p$ (then $\|A\|_* =1$), classical random matrix theory inferences allow us to estimate the spectral distribution of the robust scatter matrix from the estimation of its Stieltjes transform $m(z) = \frac{1}{p} \tr( Q_{-z}) $.
This is confirmed in Figure~\ref{fig:example} which depicts the eigenvalue distribution of the sample covariance of the data matrix $X$: (i) deprived of the influence of $\tau$ (i.e., for $\tau = I_n$), (ii) corrected with the robust scatter matrix (i.e., it is here the sample covariance matrix of the equivalent data $Xu(\hat \Delta)^{1/2}$), and (iii) without any modification on $X$. For the two first spectral distributions, we displayed their estimation with the Stieltjes transform as per Theorem~\ref{the:Estimation_transforme_stieltjes}.

We additionally can provide guarantees on the alignment of the eigen vector $v_{\text{max}}$ associated to the highest eigen value with the signal $m$, as expressed on Figure~\ref{fig:example}. Indeed if we take $A = mm^T$ (then $\|A\|_* = \|m\|^2 \leq O(1)$), for any path $\gamma$ containing the highest eigen value of $\hat C$ but no other value from the bulk we have the identity:
\begin{align*}
  \frac{1}{2i\pi}\oint_\gamma m^T Q_{-z} m dz = (v_{\text{max}}^T m)^2.
\end{align*}
Estimating $m^TQm$ therefore leads to estimating the projection $(v_{\text{max}}^T m)^2$.

\begin{figure}
  \includegraphics[width=\linewidth]{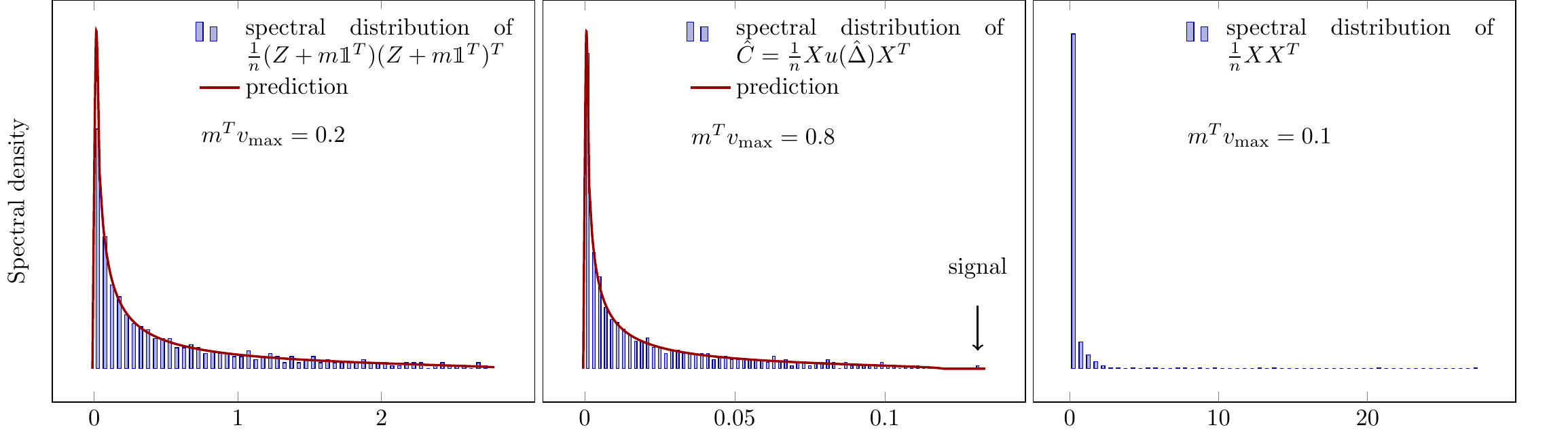}
  \caption{ Spectral distributions of the matrices $\frac{1}{n}(Z+m\un^T)(Z+m\un^T)^T$, $\hat C$ and $\frac{1}{n}XX^T$ against their large dimensional prediction; $p=500$, $n=400$ (null eigenvalues removed), $u : t \mapsto \min(t, \frac{1}{1+5t})$, the variables $\tau_1,\ldots,\tau_n$ are drawn independently from a Student distribution with $1$ degree of freedom, $m  = \un /\sqrt p\in\mathbb R^p$; $Z = \sin(W)$ for $W \sim \mathcal N(0,AA^T)$ where $A\in \mathcal M_p$ is a fixed matrix whose entries are drawn from the Gaussian distribution with zero mean and unit variance ($Z\propto \mathcal E_2$ by construction). The population covariance and mean of $Z$ are computed with a set of $p^2$ independent realizations of $Z$. The values of the projections of the signal $m$ against the eigenvector $v_{\text{max}}$ associated to the largest eigenvalue reveals that, with the robust scatter approach, the diverging action of $\tau$ in the model can be turned into an advantage to infer the signal $m$ from the data. The choice of the mapping $u$ is not optimized, our goal here is just to show that non monotonic functions are suited to robust statistics as long as they satisfy our assumptions.}
  \label{fig:example}
\end{figure} 

\section{Preliminaries for the study of the resolvent}\label{sec:Preliminaires}
Let $\mathcal S_p$ be the set of symmetric matrices of size $p$ and $\mathcal S_p^+$ the set of symmetric nonnegative matrices. Given $S,T \in \mathcal S_p$, we denote $S \leq T$ iif $T-S \in \mathcal S^+_p$. We will extensively work with the set $\left(\mathcal S_p\right)^n$ which will be denoted for simplicity $\mathcal S^n_p$. Given $S\in \mathcal S^n_p$, we finally let $S_1,\ldots, S_n \in \mathcal S_p$ be its $n$ components.

Given two sequences of scalars $a_{n,p}, b_{n,p}$, the notation $a_{n,p} \sim O(b_{n,p})$ means that $a_{n,p} \leq O(b_{n,p})$ and $a_{n,p} \geq O(b_{n,p})$.
We extend those characterizations to diagonal matrices: given $\Delta \in \mathcal D_n^+$, $\Delta \leq O(1)$ indicates that $\|\Delta \| \leq O(1)$ while $\Delta \geq O(1)$ means that $\|\frac{1}{\Delta } \| \leq O(1)$ and $\Delta \sim O(1)$ means that $O(1)\leq \Delta \leq O(1)$.

The different assumptions leading to the main results are presented progressively throughout the article so the reader easily understands their importance and direct implications. A full recollection of all these assumptions is provided at the beginning of the appendix.

\subsection{The resolvent behind robust statistics and its contracting properties}

Given $\gamma >0$ and $S\in \mathcal S_p^n$, we introduce the resolvent function at the core of our study~:
\begin{align*}
    \res_\gamma: \ \begin{aligned}[t] 
    \mathcal S_p^n \times \mathcal D_n^+ &&\longrightarrow&& \mathcal M_p  \hspace{1.4cm}&\\
    (S,\Delta) \hspace{0.2cm} && \longmapsto&&\left(\frac{1}{n}\sum_{i=1}^n \Delta_i S_i + \gamma I_p\right)^{-1}.
    \end{aligned}
\end{align*}
Given a dataset $X = (x_1,\ldots,x_n) \in \mathcal M_{p,n}$, if we note $X\cdot X^T = (x_ix_i^T)_{1\leq i\leq n} \in \mathcal S_p^n$, the robust estimation of the scatter matrix then reads (if well defined):
\begin{align}\label{eq:point_fixe_Ch_Delta}
  \hat C = \frac{1}{n} X u(\hat \Delta)X^T&
  &\text{with}&
  &\hat \Delta = \diag \left( \frac{1}{n}x_i^T \res_\gamma(X \cdot X^T,u(\hat\Delta))x_i)\right)_{1\leq i \leq n}.
\end{align}
In the following, we will denote for simplicity $\res_\gamma^X \equiv \res_\gamma(X \cdot X^T,u(\hat\Delta))$.
To understand the behavior (structural, spectral, statistical) of $\hat C$, one needs first to understand the behavior of the resolvent $\res_\gamma(S, \Delta)$ for general $S\in \mathcal S_p^n$ and $\Delta \in \mathcal D_n^+$. We document in this subsection its contracting properties.

As the scalar $\gamma$ will rarely change in the remainder, it will be sometimes omitted for readability.

\begin{lemma}\label{lem:Borne_resolvente}
  Given $\gamma >0$, $S \in \mathcal S_p^n$, $M\in \mathcal M_{p,n}$ and $\Delta \in\mathcal D_n^+$:
  \begin{align*}
    \left\Vert \res_\gamma(S,\Delta) \right\Vert \leq \frac{1}{\gamma};&
    &\left\Vert \frac{1}{\sqrt n} \res_\gamma(M \cdot M^T,\Delta) M \Delta^{\frac{1}{2}} \right\Vert \leq \frac{1}{\sqrt \gamma};&
    &\left\Vert \frac{1}{n} \res_\gamma(S,\Delta) \sum_{l=1}^k \Delta_lS_l\right\Vert \leq 1.
  \end{align*}
\end{lemma}

\begin{proof}
  We can bound in the space of symmetric matrices:
  \begin{align*}
    \frac{1}{n}\sum_{i=1}^n \Delta_i S_i + \gamma I_p \geq \gamma I_n,
  \end{align*}
  thus $Q_\gamma(S, \Delta) \leq \frac{1}{\gamma} I_n$. Besides, we can write:
  \begin{align*}
    Q_\gamma(S, \Delta) \frac{1}{n}\sum_{i=1}^n \Delta_i S_i = I_p - \gamma Q_\gamma(S, \Delta) \leq I_p.
  \end{align*}
  Noting $S_i = m_im_i^T$ and $S = (S_1, \ldots S_n) = M \cdot M^T$, we can then deduce that:
  \begin{align*}
     \frac{1}{\sqrt n} \res_\gamma(M \cdot M^T,\Delta) M \Delta^{\frac{1}{2}} = \left( Q_\gamma(S, \Delta) \frac{1}{n}\sum_{i=1}^n \Delta_i S_i Q_\gamma(S, \Delta) \right)^{1/2},
   \end{align*} 
   which provides the second bound.
\end{proof}

  Given $M \in \mathcal M_{p,n}$, and $S \in \mathcal S_p^n$, further define the mapping $I_\gamma : \mathcal S_p^n \times \mathcal D_n^+ \rightarrow  \mathcal D_n^+$, 
  \begin{align*}
    I(S,\Delta)  = \diag\left(\frac{1}{n}\tr\left(S_i \res_\gamma(S,\Delta)\right)\right)_{1\leq i \leq n}.
\end{align*}
With the notation $I^X_\gamma(\Delta) \equiv I(X\cdot X^T, \Delta)$, the fixed point $\hat \Delta$ defined in \eqref{eq:point_fixe_Ch_Delta} is simply $\hat \Delta = I^X_\gamma(u(\hat \Delta))$.
To prove the existence and uniqueness of $\hat \Delta$ we exploit the Banach fixed-point theorem to find contracting properties on the mapping $\Delta \mapsto I^X_\gamma(u( \Delta))$ for which $\hat\Delta$ is a fixed point. As we see in the following lemma, the contractive character does not appear relatively to the spectral norm on $\mathcal D_n^+$ but relatively to another metric which will be later referred to as the ``stable semi-metric''.
\begin{lemma}\label{lem:I_conctractante}
  Given $S\in \mathcal S_p^n$ and $\Delta, \Delta' \in \mathcal D_n^+$, we have 
  (the index $\gamma$ being omitted)
  \begin{align*}
     \left\Vert \frac{ I(S,\Delta) -  I(S,\Delta')}{  \sqrt{ I(S,\Delta)  I(S,\Delta')}}\right\Vert < \max \left( \phi^S_\gamma( \Delta), \phi^S_\gamma( \Delta')\right) \left\Vert\frac{\Delta -\Delta'}{  \sqrt{\Delta \Delta'}} \right\Vert,
  \end{align*}
  where $\phi^S_\gamma(\Delta) \equiv \left\Vert 1- \gamma  \res_\gamma(S,\Delta)\right\Vert$
\end{lemma}
\begin{proof}
  Given $a\in \{1,\ldots, k\}$, we can bound thanks to Cauchy Shwarz inequality:
  \begin{align*}
    \left\vert  I(S,\Delta)_a -  I(S,\Delta')_a\right\vert
    & =  \left\vert \frac{1}{n} \tr \left(S_a \left( \res_\gamma(S,\Delta') -  \res_\gamma(S,\Delta)\right)\right)\right\vert \\ 
    & =  \left\vert \frac{1}{n} \sum_{b=1}^k\tr \left(S_a  \res_\gamma(S,\Delta') S_b \left(\Delta'_b -\Delta_b\right)  \res_\gamma(S,\Delta)\right)\right\vert \\ 
    & \leq \frac{1}{n} \sqrt{  \sum_{b=1}^k\tr \left(S_a \res_\gamma(S,\Delta) \frac{S_b \left\vert\Delta'_b -\Delta_b\right\vert}{\sqrt{\Delta_b \Delta_b'}}\Delta_b  \res_\gamma(S,\Delta)\right)} \\ 
    & \hspace{1cm} \cdot  \sqrt{  \sum_{b=1}^k\tr \left(S_a  \res_\gamma(S,\Delta') \frac{S_b \left\vert\Delta'_b -\Delta_b\right\vert}{\sqrt{\Delta_b \Delta_b'}}\Delta'_b  \res_\gamma(S,\Delta')\right)} \\ 
    & \leq \left\Vert \frac{\Delta' -\Delta}{\sqrt{\Delta \Delta'}}\right\Vert  \sqrt{\frac{1}{n}\tr \left(S_a \res_\gamma(S,\Delta)\left(1- \gamma  \res_\gamma(S,\Delta)\right)\right)} \\ 
    & \hspace{1cm} \cdot  \sqrt{\frac{1}{n}\tr \left(S_a \res_\gamma(S,\Delta')\left(1- \gamma  \res_\gamma(S,\Delta')\right)\right)}\\
    &\leq \sqrt{\phi^S_\gamma(\Delta) \phi^S_\gamma(\Delta')} \left\Vert \frac{\Delta' -\Delta}{\sqrt{\Delta \Delta'}}\right\Vert \sqrt { I(S,\Delta)_a I(S,\Delta')_a}
  \end{align*}
\end{proof}
If one sees the term $\left\Vert \frac{\Delta- \Delta'}{\sqrt{\Delta \Delta'}}\right\Vert$ as a distance between $\Delta$ and $\Delta'$, then Lemma~\ref{lem:I_conctractante} sets the $1$-Lipschitz character of $I(S,\cdot) : \Delta \mapsto I(S, \Delta)$, which is a fundamental property in what follows. 
We present in the next subsection a precise description of such functions that will be called \textit{stable mappings}.

\subsection{The stable semi-metric}\label{classe_stable_R}

The stable semi-metric which we define here is a convenient object which allows us to set Banach-like fixed point theorems. It has a crucial importance to prove the existence and uniqueness of $\hat C$ but also to obtain some random matrix identities on $\hat C$, such as the estimation of its limiting spectral distribution. 
\begin{definition}\label{def:distance_stable}
  We call the \emph{stable semi-metric} on $\mathcal D_n^+ = \{D\in \mathcal D_n, \forall i \in [n], \ D_i>0\}$ the function:
  \begin{align}\label{eq:definition_stable_metric}
    \forall \Delta,\Delta' \in \mathcal D_n^+ : \ \ d_s(\Delta, \Delta') \equiv  \left\Vert  \frac{ \Delta- \Delta'}{\sqrt{\Delta \Delta'}}\right\Vert.
  \end{align}
\end{definition}

In particular, this semi-metric can be defined on $\mathbb R^+$, identifying $\mathbb R^+$ with $\mathcal D_1^+$. 

The function $d_s$ is not a metric because it does not satisfy the triangular inequality, one can see for instance that:
  \begin{align*}
    d_s(4,1) = \frac{3}{2} > \frac{1}{\sqrt 2} + \frac{1}{\sqrt 2}=d_s(4,2) + d_s(2,1)  
  \end{align*}
  More precisely, for any $x,z \in \mathbb R_+$ such that $x<z$, if one differentes twice the mapping $g : y \to \frac{(xy-x)^2}{xy} + \frac{(z-y)^2}{xy}$, one obtains:
   \begin{align*}
     g'(y) = \frac{1}{y}- \frac{y}{x^2} + \frac{1}{z} - \frac{z}{x^3}&
     &\text{and}&
     &g'\!{}'(y) = \frac{3y}{x^3} + \frac{3z}{x^3}>0,
   \end{align*}
  which proves that $g$ is strictly convex on $[x,z]$ and therefore it admits a minimum $y_0$ on $]x,z[$ (since $g(x) = g(z)$). In particular, one can bound:
  \begin{align*}
    d_s(x,z ) > \sqrt{d_s(x,y_0)^2 + d_s(y_0,z)^2}
  \end{align*}
  One can however sometimes palliate this weakness when needed thanks to the following inequality proved in Appendix~\ref{app:stable_metic}.
  \begin{proposition}[Pseudo triangular inequality]\label{pro:palliation_absence_inegalite_triangulaire}
    Given $x,z,y \in \mathbb R^+$:
    \begin{align*}
        |x-y| \leq |y - z|&
        &\Longrightarrow&
       &d_s(x,y) \leq d_s(x,z).
     \end{align*}
     In addition, for any $p\in \mathbb N^*$ and $y_1,\ldots,y_{p-1} \in \mathbb R^+$, we have the inequalities\footnote{The mapping $x\mapsto x^p$ is not Lipschitz for the semi-metric $d_s$ (unlike $x \mapsto x^{\frac{1}{p}}$).}:
  \begin{align*}
    d_s(x,y_1) + \cdots + d_s(y_{p-1},z) \geq d_s \left(x^{\frac{1}{p}},z^{\frac{1}{p}}\right)
    \geq d_s \left(x,z\right)^{1/p}
  \end{align*}
  and the left inequality turns into an equality in the case $y_i = x^{\frac{p-i}{p}} z^{\frac{i}{p}}$ for $i\in\{1,\ldots,p-1\}$.
  \end{proposition}

The semi-metric $d_s$ is called stable due to its many interesting stability properties.
\begin{property}\label{prot:stabilite_homotethie_distance_stable}
  Given $\Delta,\Delta' \in \mathcal D_n^+$ and $\Lambda \in \mathcal D_n^+$:
  \begin{align*}
     d_s \left(\Lambda \Delta, \Lambda \Delta'\right) 
    =d_s \left(\Delta, \Delta'\right)&
    &\text{and}&
    &d_s \left(\Delta^{-1}, \Delta'^{-1}\right) =d_s \left(\Delta, \Delta'\right).
  \end{align*}
\end{property}

\begin{property}\label{prot:pseudo_tringular_inequality}
  Given four diagonal matrices $\Delta,\Delta',D,D' \in \mathcal D_n^+$:
  \begin{align*}
     d_s(\Delta +D,  \Delta'+D') \leq \max(d_s(\Delta,  \Delta'), d_s(D, D')).
   \end{align*} 
\end{property}
To prove this property one needs two elementary results.
\begin{lemma}\label{lem:stability_preliminary_lemma}
  Given four positive numbers $a,b,\alpha, \beta \in \mathbb R^+$:
  \begin{align*}
    \sqrt{ab} + \sqrt{\alpha \beta} \leq \sqrt{(a+\alpha)(b+\beta)}&
    &\text{and}&
    &\frac{a+\alpha}{b+\beta} \leq \max \left(\frac{a}{b} ,\frac{\alpha}{\beta} \right)
  \end{align*}
\end{lemma}
\begin{proof}
  For the first result, we deduce from the inequality $4ab\alpha\beta \leq (a\alpha + b\beta)^2$:
  \begin{align*}
    \left(\sqrt{ab} + \sqrt{\alpha \beta}\right)^2 
    &= ab + \alpha \beta + 2\sqrt{ab\alpha\beta}
    \leq ab + \alpha \beta + a\beta + b\alpha = (a+\alpha)(b+\beta)
  \end{align*}
  For the second result, we simply bound:
  \begin{align*}
    \frac{a+\alpha}{b+\beta} \leq \frac{a}{b} \frac{b}{b+\beta} + \frac{\alpha}{\beta} \frac{\beta}{b+\beta} \leq \max \left(\frac{a}{b} ,\frac{\alpha}{\beta} \right) \left( \frac{b}{b+\beta} + \frac{\beta}{b+\beta} \right) = \max \left(\frac{a}{b} ,\frac{\alpha}{\beta} \right)
  \end{align*}
\end{proof}
\begin{proof}[Proof of Property~\ref{prot:pseudo_tringular_inequality}]
  For any $\Delta,\Delta',D,D' \in \mathcal D_n^+$, there exists $i_0 \in [n]$ such that:
  \begin{align*}
  d_s(\Delta+D, \Delta'+D')
    &= \frac{\left\vert \Delta_{i_0}- \Delta_{i_0}'+ D_{i_0}-D_{i_0}' \right\vert}{\sqrt{(\Delta_{i_0}+D_{i_0})(\Delta'_{i_0}+D_{i_0}')}}\\
    &\leq \frac{\left\vert \Delta_{i_0}-\Delta_{i_0}'\right\vert +\left\vert D_{i_0}-D_{i_0}' \right\vert}{\sqrt{\Delta_{i_0}\Delta'_{i_0}} +\sqrt{D_{i_0}D_{i_0}' })}
    \leq \max \left(\frac{\left\vert \Delta_{i_0}-\Delta_{i_0}'\right\vert}{\sqrt{\Delta_{i_0}\Delta_{i_0}'}}, \frac{\left\vert D_{i_0}-D_{i_0}' \right\vert}{\sqrt{D_{i_0}D_{i_0}'}}\right)
  \end{align*}
  thanks to Lemma~\ref{lem:stability_preliminary_lemma}.
\end{proof}
\subsection{The stable class}\label{sse:stable_class}
\begin{definition}[Stable class]\label{def:classe_des_fonctions_stables}
  The set of $1$-Lipschitz functions for the stable semi-metric is called the stable class. We denote it:
  \begin{align*}
    \mathcal S \left(\mathcal D_n^+\right) \equiv \left\{ f : \mathcal D_n^+ \rightarrow \mathcal D_n^+ \ | \ \forall \Delta,\Delta' \in \mathcal D_n^+,  \ \Delta \neq \Delta': \ d_s(f(\Delta),f(\Delta')) \leq d_s (\Delta,\Delta')\right\}.
  \end{align*}
  The elements of $\mathcal S \left(\mathcal D_n^+\right) $ are called the stable mappings. 
\end{definition}

Let us then provide the properties which justify why we call $\mathcal S \left(\mathcal D_n^+\right)$ a \textit{stable} class: this class indeed satisfies far more stability properties than the usual Lipschitz mappings (for a given norm). Those stability properties are direct consequences to Properties~\ref{prot:stabilite_homotethie_distance_stable} and~\ref{prot:pseudo_tringular_inequality}.
\begin{property}\label{prot:stabilite_contr_mapping_mat_diagonales}
  Given $\Lambda, \Gamma\in \mathcal D_n^+$ and $f,g \in \mathcal S(\mathcal D_n^+)$:
  \begin{align*}
    \left( \Delta \mapsto \Lambda f(\Gamma \Delta) \right)\in \mathcal S(\mathcal D_n^+),&
    &\frac{1}{f}\in \mathcal S(\mathcal D_n^+),&
    &f \circ g\in \mathcal S(\mathcal D_n^+),&
    &f +g\in \mathcal S(\mathcal D_n^+).
  \end{align*}
\end{property}
\subsection{The sub-monotonic class}
The stable class has a very simple interpretation when $n=1$.
Given a function $f: \mathbb R^+ \rightarrow \mathbb R^+$ we introduce two characteristic functions $f_/,f_\times : \mathbb R^+ \rightarrow \mathbb R^+$:
\begin{align*}
    f_{/} \ : \ x \mapsto \frac{f(x)}{x} &
    &\text{and}&
    &f_{\times} \ : \ x \mapsto x f(x).
  \end{align*}
\begin{property}\label{prot:Caracterisation_stabilite_monotonie}
A function $f : \mathbb R^+ \rightarrow \mathbb R^+$ is a stable mapping if and only if $f_/ $ is non-increasing and $f_\times$ is non-decreasing.
\end{property}
\begin{proof}
  
  Let us consider $x,y \in \mathbb R^+$, such that, say, $x\leq y$. We suppose in a first time that $f_/$ is non-increasing and that $f_\times$ is non-decreasing.
  We know that $\frac{f(x)}{x} \geq \frac{f(y)}{y}$, and subsequently:
  \begin{align}\label{eq:inegalite_caract_syabe_class1}
    f(y) - f(x) \leq \frac{f(y)}{y}(y-x)&
    &\text{and}&
    &f(y) - f(x) \leq \frac{f(x)}{x}(y-x)
  \end{align}
  The same way, since $f(x)x \leq f(y)y$ we also have the inequalities:
  \begin{align}\label{eq:inegalite_caract_syabe_class2}
    f(x)-f(y) \leq \frac{f(y)}{x}(y-x)&
    &\text{and}&
    &f(x) - f(y) \leq \frac{f(x)}{y}(y-x)
  \end{align}
  Now if $f(y)\geq f(x)$, we can take the root of the product of the two inequalities of \eqref{eq:inegalite_caract_syabe_class1} and if $f(y)\leq f(x)$, we take the root of the product of the two inequalities of \eqref{eq:inegalite_caract_syabe_class2}, to obtain, in both cases:
  \begin{align*}
    \left\vert f(x)-f(y)\right\vert \leq \sqrt{\frac{f(y)f(x)}{xy}} \left\vert x-y\right\vert
  \end{align*}
  That means that $f\in \mathcal S(\mathbb R^+)$.

  Conversely, if we now suppose that $f\in \mathcal S(\mathbb R^+)$, we then use the bound:
  \begin{align*}
    \left\vert f(y) - f(x) \right\vert \leq \sqrt{\frac{f(y)f(x)}{xy}} (y-x).
  \end{align*}
  First, if $f(x) \leq f(y)$, then $f(x)x \leq f(y)y$ and we can bound:
  \begin{align*}
   f(y)-f(x)
   \leq\max \left( \frac{f(x)}{x}, \frac{f(y)}{y} \right) (y-x)
   \leq \max \left( \left( \frac{y}{x} - 1  \right)f(x), \left( 1-\frac{x}{y} \right) f(y) \right) &
  \end{align*}
  which directly implies $\frac{f(y)}{y}\leq \frac{f(x)}{x}$.
  Second, if $f(x) \geq f(y)$, $\frac{f(x)}{x} \geq \frac{f(y)}y$ and we can then bound in the same way:
  \begin{align*}
    f(x)xy- f(y)xy\leq \max \left( xf(x)(y-x), \left( y-x \right) yf(y) \right)
  \end{align*}
  which implies $xf(x) \leq yf(y)$.
  In both cases ($f(x) \leq f(y)$ and $f(y) \leq f(x)$), and we see that $f_/(x) \geq f_/(y)$ and $f_\times(x) \leq f_\times(y)$, proving the result.
  
\end{proof}

This property allows us to understand directly that the stability of a function is a local behavior. We then conclude staightforwardly that the supremum or the infimum of stable mappings is also stable.
\begin{corollary}\label{cor:inf_sup_stable}
  Given a family of stable mappings $(f_\theta)_{\theta \in \Theta} \in \mathcal S(\mathbb R^+)^\Theta$, for a given set $\Theta$, the mappings $\sup_{\theta \in \Theta} f_\theta$ and $\inf_{\theta \in \Theta} f_\theta$ are both stable.
\end{corollary}

Given $f: \mathcal D^+_n \rightarrow \mathcal D^+_n$, we can introduce by analogy to the case of mappings on $\mathbb R^+$, the mappings $f_/, f_\times: \mathcal D^+_n \rightarrow \mathcal D^+_n$ defined with:
\begin{align*}
    f_{/} \ : \ \Delta \mapsto \tr  \left(\frac{f(\Delta)}{\Delta}\right) &
    &\text{and}&
    &f_{\times} \ : \ \Delta \mapsto \tr  \left(\Delta f(\Delta)\right)
\end{align*}
Inspiring from Property~\ref{prot:Caracterisation_stabilite_monotonie}, one can then define:
\begin{definition}[Sub-monotonuous class]\label{def:sub_monotonuous_class}
   A mapping $f: \mathcal D_n^+ \to \mathcal D_n^+$ is said to be sub-monotonuous if and only if $f_\times$ is non decreasing and $f_/$ is non-increasing, we note this class of mappings $ \mathcal S_m(\mathcal D_n^+)$.
 \end{definition} 
 \begin{remark}\label{rem:sub_monotonuous_and_stable_class}
   We know from Property~\ref{prot:Caracterisation_stabilite_monotonie} that $\mathcal S(\mathbb R_+)= \mathcal S_m(\mathbb R_+)$ but for $n>1$, none of the classes $\mathcal S_m(\mathcal D_n^+)$ and $\mathcal S(\mathcal D_n^+)$ contains strictly the other one. On the first hand, introducing:
   \begin{align*}
     f: \begin{aligned}[t]
       \mathcal D_2^+ & \longrightarrow& \mathcal D_2^+ \hspace{1.5cm}\\
       \Delta \ & \longmapsto& \diag \left(\frac{1}{\Delta_2}, \frac{1}{\Delta_1}\right),
     \end{aligned}
   \end{align*}
   we see that $f \in \mathcal S(\mathcal D_n^+)$ but $f \notin \mathcal S_m(\mathcal D_n^+)$ (for $\Delta = \diag(1,2)$ and $\Delta' = \diag(2, 2)$, $\Delta \leq \Delta'$ but $\tr(f(\Delta)\Delta) = \frac{5}{2} > 2 = \tr(f(\Delta')\Delta')$). On the other hand, the mapping:
   \begin{align*}
     g: \begin{aligned}[t]
       \mathcal D_2^+ & \longrightarrow& \mathcal D_2^+ \hspace{1.5cm}\\
       \Delta \ & \longmapsto& \diag \left(\frac{\Delta_1\Delta_2}{1+\Delta_2}, 1\right)
     \end{aligned}
   \end{align*}
   is in $\mathcal S_m(\mathcal D_n^+)$ because:
   \begin{align*}
     \left\{\begin{aligned}
       \frac{\partial g_.}{\partial \Delta_1} &= \frac{2\Delta_1\Delta_2}{1+\Delta_2} \geq 0\\
       \frac{\partial g_.}{\partial \Delta_2} &= \frac{\Delta_1^2}{(1+\Delta_2)^2} + 1 \geq 0
     \end{aligned}\right.&
     &\text{and}&
     &\left\{\begin{aligned}
       \frac{\partial g_/}{\partial \Delta_1} &= 0\leq 0\\
       \frac{\partial g_/}{\partial \Delta_2} &= \frac{1}{(1+\Delta_2)^2} - \frac{1}{\Delta_2^2} \leq 0.
     \end{aligned}\right.
   \end{align*}
   However, we can see that $g$ is not stable if we introduce the diagonal matrices $\Delta = \diag(1,2)$ and $\Delta' = \diag(2, 3)$ since we have then:
   \begin{align*}
     d_s(g(\Delta), g(\Delta')) = d_s \left(\frac{2}{3},\frac{3}{2}\right)  = \frac{5}{6} > \frac{1}{\sqrt 2} = \max \left(\frac{|3-2|}{\sqrt {6}}, \frac{|2-1|}{\sqrt 2}\right) = d_s(\Delta, \Delta')
   \end{align*}
 \end{remark}

\subsection{Fixed Point theorem for stable and sub-monotonic mappings}

The Banach fixed point theorem states that a contracting function on a complete space admits a unique fixed point. The extension of this result to contracting mappings on $\mathcal D_n^+$, for the semi-metric $d_s$, is not obvious: first because $d_s$ does not verify the triangular inequality and second because the completeness needs be proven. The completeness of the semi-metric space $(\mathcal D_n^+,d_s)$ is left in Appendix~\ref{app:stable_metic} since we will not need it.
Let us start with a first bound.
\begin{lemma}\label{lem:suite_iterative_application_contractante_borne}
  Given a mapping $f : \mathcal D_n^+ \to \mathcal D_n^+$, contracting for the semi metric $d_s$, any sequence of diagonal matrices $(\Delta_n)_{n \in \mathbb N}$ satisfying $\Delta^{(p+1)} = f(\Delta^{(n)})$ is bounded from below and above and satisfies for all $p\in \mathbb N$ and $i \in [n]$:
  \begin{align*}
     \exp \left(- \frac{\lambda d_s \left( \Delta^{(1)}, \Delta^{(0)} \right)}{2(1- \lambda)} \right) \Delta^{(1)}_i
     \leq \Delta^{(p)}_i 
     \leq  \exp \left( \frac{\lambda d_s \left( \Delta^{(1)}, \Delta^{(0)} \right)}{2(1- \lambda)} \right) \Delta^{(1)}_i
  \end{align*}
\end{lemma}
\begin{proof}
  Noting $\lambda >0$, the Lipschitz parameter of $f$ for the semi-metric $d_s$, let us first show that, for all $i\in [n]$:
  \begin{align}\label{eq:borne_Delta_np1_s_Delta_n}
     \forall p\in \mathbb N, \sqrt{\frac{\Delta^{(p+1)_i}}{\Delta^{(p)}_i}} \leq  1 + \lambda^p d_s \left( \Delta^{(1)}, \Delta^{(0)} \right).
  \end{align} 
  When $\Delta^{(p+1)}_i \leq \Delta^{(p)}_i$, it is obvious and when $\Delta^{(p+1)} \geq \Delta^{(p)}$ the contractivity of $f$ allows us to set $d_s(\Delta^{(p+1)}, \Delta^{(p)}) \leq \lambda^p d_s \left( \Delta^{(1)}, \Delta^{(0)} \right)$, which implies:
  \begin{align*}
    \sqrt{\frac{\Delta^{(p+1)}_i}{\Delta^{(p)}_i}} \leq \sqrt{\frac{\Delta^{(p)}_i}{\Delta^{(p+1)}_i}} + \lambda^p d_s \left( \Delta^{(1)}, \Delta^{(0)} \right) \leq 1 + \lambda^p d_s \left( \Delta^{(1)}, \Delta^{(0)} \right).
  \end{align*}
  Multiplying \eqref{eq:borne_Delta_np1_s_Delta_n} for all $p \in \{1,\ldots, P\}$, we obtain:
  \begin{align*}
    \sqrt{\frac{\Delta^{(P)}_i}{\Delta_i^{(1)}}} 
    &\leq 1 \leq \prod_{p=1}^P \left( 1 + \lambda^p d_s \left( \Delta^{(1)}, \Delta^{(0)} \right) \right) 
    = \exp \left( \sum_{p=1}^P \log \left( 1+ \lambda^p d_s \left( \Delta^{(1)}, \Delta^{(0)} \right) \right)\right)\\
    &\leq \exp \left( \sum_{p=1}^P  \lambda^p d_s \left( \Delta^{(1)}, \Delta^{(0)} \right)\right) \leq \exp \left( \frac{\lambda d_s \left( \Delta^{(1)}, \Delta^{(0)} \right)}{1- \lambda} \right).
  \end{align*}
  With a similar approach, we can eventually show the result of the lemma.  
\end{proof}

Let us now present two fixed point results that will justify the definition of the robust scatter matrix $C$ but also of the deterministic diagonal matrix $U$ introduced in Section~\ref{sec:main_result}. 

\begin{theorem}\label{the:point_fixe_fonction_contractante_de_D_n_borne_superieurement}
  Any mapping $f: \mathcal D_n^+\to \mathcal D_n^+$, contracting for the stable semi-metric $d_s$, admits a unique fixed point $\Delta^*\in\mathcal D_n(\mathbb R^+ \cup \{0\})$ satisfying $\Delta^* =  f(\Delta^*)$.
\end{theorem}
\begin{proof}
  We cannot repeat exactly the proof of the Banach fixed point theorem since $d_s$ does not satisfy the triangular inequality. 
  Noting $\lambda \in (0,1)$ the parameter such that $\forall \Delta,\Delta' \in \mathcal D_n^+$, $d_s(f(\Delta) , f(\Delta')) \leq \lambda d_s(\Delta,\Delta')$, we show that the sequence $(\Delta^{(k)})_{k\geq 0}$ satisfying: 
  \begin{align*}
    \Delta^{(0)} = I_n&
    &\text{and}&
    &\forall k\geq 1: \ \Delta^{(k)} = f(\Delta^{(k-1)})
   \end{align*}
   is a Cauchy sequence in $(\overline{\mathcal D}_n^+ ,\|\cdot \|)$, where $\overline{\mathcal D}_n^+ \equiv \mathcal D_n(\mathbb R^+ \cup \{0\})$.

 We know from Lemma~\ref{lem:suite_iterative_application_contractante_borne} that there exists $\delta>0$ such that $\forall p \in \mathbb N$, $\Vert \Delta^{(p)} \Vert \leq \delta$. One can then bound for any $p \in \mathbb N$:
  \begin{align*}
      \left\Vert \Delta^{(p+1)}-\Delta^{(p)}\right\Vert\leq \delta
      d_s(\Delta^{(p+1)},\Delta^{(p)})
      \leq \lambda^p \delta d_s(\Delta^{(1)},\Delta^{(0)}).
  \end{align*}
  Therefore, thanks to the triangular inequality (in $(\mathcal D_n^+ ,\|\cdot \|)$), for any $n \in \mathbb N$:
  \begin{align*}
    \left\Vert\Delta^{(p+n)} - \Delta^{(p)}\right\Vert
    &\leq \left\Vert\Delta^{(p+n)} - \Delta^{(p+n-1)}\right\Vert+\cdots + \left\Vert\Delta^{(p+1)} - \Delta^{(p)}\right\Vert \\
    &\leq \  \frac{\delta d_s(\Delta^{(1)},\Delta^{(0)})}{1-\lambda} \lambda^p \ \ \ \underset{p \to \infty}{\longrightarrow} \ 0.
  \end{align*}
  That allows us to conclude that $(\Delta^{(p)})_{p\in \mathbb N}$ is a Cauchy sequence, and therefore that it converges to a diagonal matrix $\Delta^* \equiv \lim_{p \to \infty} \Delta^{(p)} \in \overline{\mathcal D}_n^+$ which is complete (closed in a complete set). But since $\Delta^{(p)}$ is bounded from below and above thanks to Lemma~\ref{lem:suite_iterative_application_contractante_borne}, we know that $\Delta^* \in \mathcal D_n^+$. By contractivity of $f$, it is clearly unique.
\end{proof}
It is possible to relax a bit the contracting hypotheses on $f$ if one supposes that $f$ is monotonic. We express rigorously this result in next theorem, but it will not be employed in our paper since we preferred to assume $u$ bounded to obtain the contracting properties of the fixed point satisfied by $\hat \Delta$. The proof is left in Appendix~\ref{app:stable_metic}.
\begin{theorem}\label{the:point_fixe_fonction_stable_monotone_de_D_n}
   Let us consider a weakly monotonic mapping $f: \mathcal D_n^+\to \mathcal D_n^+$ bounded from below and above. If we suppose that $f$ is stable and verifies:
  \begin{align}\label{eq:faiblement_contractant}
    \forall \Delta,\Delta' \in\mathcal D_n^+: \ d_s(f(\Delta),f(\Delta') )<d_s(\Delta,\Delta')
   \end{align} then there exists a unique fixed point $D\in\mathcal D_n^+$ satisfying $\Delta^* = f(\Delta^*)$.
\end{theorem}

To employ Theorem~\ref{the:point_fixe_fonction_contractante_de_D_n_borne_superieurement} to the fixed point equation satisfied by $\hat \Delta$, a first step is to look at $\Delta \mapsto I_\gamma(S,\Delta) = \diag(\tr(S_i \res(S,\Delta)))_{1\leq i \leq n}$ that we know to be stable from Lemma~\ref{lem:I_conctractante}. The control on the Lipschitz parameter required by Theorem~\ref{the:point_fixe_fonction_contractante_de_D_n_borne_superieurement} is issued from the following preliminary Lemma.
\begin{lemma}\label{lem:tR_borne}
  Given $S \in \mathcal S_p^n$ and a function $f : \mathcal D_n^+ \rightarrow \mathcal D_n^+$ bounded by $f_0 \in \mathcal D_n ^+$,
  \begin{align*}
    \forall \Delta \in \mathcal D_n^+: \ \frac{I_p}{f_0\left\Vert S\right\Vert  + \gamma} \leq  \res(S,f(\Delta)) \leq \frac{I_p}{\gamma}&
    &\text{where}&
    &\Vert S\Vert \equiv \frac{1}{n}\left\Vert\sum_{i=1}^n S_a\right\Vert.
  \end{align*}
\end{lemma}

Combined with Lemma~\ref{lem:I_conctractante}, this result allows us to build a family of contracting stable mappings with the composition $\tilde I(S,\cdot) \circ f$ when $f\in \mathcal S(\mathcal D_n^+)$ is bounded from above. We thus obtain the following corollary to Theorem~\ref{the:point_fixe_fonction_contractante_de_D_n_borne_superieurement}.
\begin{corollary}\label{cor:point_fixe_pour_tI}
  Given $f\in\mathcal S(\mathcal D_n^+)$, and a family of non-negative and non-zero symmetric matrices $S = (S_1,\ldots, S_k) \in \mathcal S_{p}^n$, noting $\lambda_f$ the Lipschitz parameters of $f$ for the semi-metric $d_s$, if we assume that $f$ is bounded from above or that $\lambda_f <1$, then the fixed point equation 
  \begin{align*}
    \Delta =   I^S(f(\Delta))
  \end{align*}
  admits a unique solution in $\mathcal D_n^+$.
\end{corollary}
\begin{remark}\label{rem:contre_exemple_u_non_contractante_non_borne}
  To give a counterexample when $f$ is nor contractive for the semi-metric, nor bounded from above, let us consider the mapping $f : t \mapsto \frac{1}{t}$ and the sequence of matrix $S = (I_p, \ldots, I_p)$, then we have the equivalence:
  \begin{align*}
    \Delta = I^S(f(\Delta))&
    &\Longleftrightarrow&
    &\forall i \in [n], \gamma \Delta_i + 1 =\frac{p}{n}
  \end{align*}
  which cannot be satisfied for any $\Delta \in \mathcal{D}_{n}^+$ when $p<n$ (when $p \geq n$, the existence and uniqueness of the solution can be shown using the contractivity of $I^S \circ f$ for the spectral norm -- but this is another problem).
\end{remark}
\begin{proof}
  We can first deduce from Lemma~\ref{lem:I_conctractante} that: 
  \begin{align*}
     d_s\left( I^S(f(\Delta)), I^S(f(\Delta'))\right)  <  \lambda_f \max \left( \phi^S(f(\Delta)), \phi^S(f(\Delta')) \right) d_s\left(\Delta,\Delta' \right)&
  \end{align*}
  When $\lambda_f \geq 1$ but $\forall \Delta \in \mathcal D_n^+, f(\Delta) \leq f_0 < \infty$, we can still deduce the contractivity of $I^S \circ f$ thanks to Lemma~\ref{lem:tR_borne} that allows us to bound:
  \begin{align*}
    \max \left( \phi^S(f(\Delta)), \phi^S(f(\Delta')) \right)\leq  \sup_{\Delta \in \mathcal D_n^+} \left\Vert 1- \gamma  \res^S(f(\Delta))\right\Vert \leq \frac{1}{1+\frac{\gamma}{f_0\left\Vert S\right\Vert} }<1,
   \end{align*} 
   and conclude thanks to Theorem~\ref{the:point_fixe_fonction_contractante_de_D_n_borne_superieurement}.
\end{proof}
We will thus suppose from here on that $u$ is a bounded\footnote{We will see later that we need also to assume that $u_\times : t \mapsto t u(t)$ is bounded which implies that $u$ behaves on the infinity as a mapping $t \mapsto \frac{\alpha}{t}$ which is not contractive for the semi-metric $d_s$. Still, to be able to apply Corollary~\ref{cor:point_fixe_pour_tI}, we need to assume that $u$ is bounded.} stable function, to be able to use Corollary~\ref{cor:point_fixe_pour_tI} and set the existence and uniqueness of $\hat \Delta$ and $\hat C$ as defined in \eqref{eq:point_fixe_Ch_Delta}.
\begin{assumption}\label{ass:u_stable_bornee}
  $u\in \mathcal S(\mathbb R^+)$ and there exists $u^\infty>0$ such that $\forall t \in \mathbb R^+,  \ u(t) \leq u^\infty$. 
\end{assumption}
\begin{proposition}\label{pro:existence_unicite_D}
  For $X\in \mathcal M_{p,n}$, there exists a unique diagonal matrix $ \hat \Delta\in \mathcal D_n^+$ such that
  \begin{align*}
    \hat \Delta =  I^{X} \left(u (\hat \Delta)\right).
  \end{align*}
\end{proposition}
Now that $\hat \Delta$, and therefore $\hat C$ are perfectly defined, let us introduce additional assumptions to be able to infer concentration properties on $\hat \Delta$.

\subsection{The concentration of measure framework and Deterministic equivalent of the resolvent}\label{sse:concentration_de_la_mesure}
Having proved the existence and uniqueness of $\hat C$, we now introduce statistical conditions on $X$ to study $\hat C$ in the large dimensional $n,p\to\infty$ limit. We first define $n$ $p$-dimensional random vectors $(z_1,\ldots,z_n)\in\mathbb R^p$.
\begin{assumption}\label{ass:independance_z_i}
  The random vectors $z_1,\ldots,z_n$ are all independent.
\end{assumption}
We denote their means $\mu_i \equiv \mathbb E[z_i] \in \mathbb R^p$, their second order statistic matrix (or non-centered covariance matrix) $C_i \equiv \mathbb E[z_iz_i^T]$ and their covariance matrices $\Sigma_i = C_i-\mu_i\mu_i^T \in \mathcal M_p$.
\medskip

Let us now introduce the fundamental definition of a so-called \emph{concentrated random vector} which will allow us to obtain our estimations and concentration rates. The main idea is that a concentrated vector $W\in E$ is not ``concentrated around a point'' (visualize for instance Gaussian vectors which rather lie close to a sphere) but has concentrated ``observations'', that is random outputs $f(W)$ for \emph{any} $1$-Lipschitz map $f: E \to \mathbb R$.

To measure the speed of concentration we generally express it with the dimension of the vector space $E$ of $W$. In our particular case the dimension of interest will be the number of data $n$, we will then consider that the dimension $p=p_n$ is a function of $n$. We then assume that $n$ controls the values of $p_n$:
\begin{assumption}\label{ass:n_p_comensurables}
  $p = p_n \leq O(n)$.
\end{assumption}

To be precise, we do not define the "concentration of a random vector" but the "concentration of a sequence of random vectors". Those random vectors belong to a sequence of vector spaces that will be $(\mathcal{M}_{p_n,n})_{n\in \mathbb N}$ for $Z$, $(\mathcal{M}_{p_n})_{n\in \mathbb N}$ for $Q_z$ and $(\mathcal{D}^+_{n})_{n\in \mathbb N}$ for $\hat D$.
\begin{definition}\label{def:concentrated_sequence}
  Given a sequence of normed vector spaces $(E_n, \Vert \cdot \Vert_n)_{n \in \mathbb N}$, a sequence of random vectors $(W_n)_{n \in \mathbb N} \in \prod_{n \in \mathbb N} E_n$, a sequence of positive reals $(\sigma_n)_{n \in \mathbb N} \in \mathbb R_+ ^{\mathbb N}$ and a parameter $q>0$, we say that $W_n$ is \emph{$q$-exponentially concentrated} with an \emph{observable diameter} of order $O(\sigma_n)$ iff there exist two constants $C, c >0$ such that, for all $n \in \mathbb N$, for any $1$-Lipschitz mapping $f : E_n \rightarrow \mathbb{R}$, 
     \begin{align*}
     \forall t>0:
      \mathbb P \left(\left\vert f(W_n) - \mathbb E[f(W_n)]\right\vert\geq t\right) \leq C e^{(t/c\sigma_n)^q}
     \end{align*}
   We denote in that case $W_n \propto \mathcal E_{q}(\sigma_n)$ (or more simply $Z \propto \mathcal E_{q}(\sigma)$). If $\sigma_n \leq O(1)$\footnote{The notation $a_n =O(b_n)$ signifies that there exists a constant $K$ (independent of $s$) such that $\forall s \in S$, $a_n \leq K b_n$. The same way, $a_n \geq O(b_n)$ means that there exists a constant $\kappa>0$ such that $a_n \geq \kappa b_n$ and the notation $a_n \sim O(b_n)$ is equivalent to $a_n \leq O(b_n)$ and $a_n \geq O(b_n)$}, one can further write $W_n \propto \mathcal E_{q}$.
\end{definition}
The essential result which motivates the definition is the concentration of Gaussian vectors.
\begin{theorem}[\citep{LED05}]\label{the:concentration_vecteur_gaussien}
  Given any integer sequence $d =(d_n)_{n \in \mathbb N} \in \mathbb N^{\mathbb N}$:
  \begin{align*}
    W\sim \mathcal N(0,I_{d})&
    &\Longrightarrow&
    &W \propto \mathcal E_2
  \end{align*}
\end{theorem}
This class of random vectors is stable through Lipschitz maps:
\begin{proposition}\label{pro:stability_lipschitz_maps}
Given two sequence of normed vector spaces $(E_1, \|\cdot \|_1)$ and $(E_2,\|\cdot \|_2)$, a sequence of random vectors $W \in E_1$, two sequences $\sigma, \lambda \in \mathbb R_+$ and a sequence of $O(\lambda)$-Lipschitz function $\phi : E_1 \to E_2$:\footnote{The statement ``$\phi$ is $O(\lambda)$-Lipschitz'' means here that there exists $K \leq O(1) $ such that, for all $s \in S$, $\phi_s$ is $(K\lambda_s)$-Lipschitz.}
\begin{align*}
  W \propto \mathcal E_q(\sigma)& &\Longrightarrow& &\phi(W) \propto \mathcal E_q(\lambda \sigma).
\end{align*}
\end{proposition}
This property is not satisfied by a weaker kind of concentration, called the \textit{linear concentration}, for which only the linear observations are concentrated. 
In this paper we will obtain the linear concentration of the resolvent $Q_z$ from an hypothesis of Lipschitz concentration for $Z$.
\begin{definition}\label{def:concentrated_sequence}
  Given a sequence of normed vector spaces $(E_s, \Vert \cdot \Vert_n)_{n \in \mathbb N}$, a sequence of random vectors $(Z_n)_{n \in \mathbb N} \in \prod_{n \in \mathbb N} E_n$, a sequence of positive reals $(\sigma_n)_{n \in \mathbb N} \in \mathbb R_+ ^{\mathbb N}$ and a parameter $q>0$, we say that $Z_n$ is \emph{linearly concentrated} with an \emph{observable diameter} of order $O(\sigma_n)$ around a \emph{deterministic equivalent} $\tilde Z_n \in E_n$ iff there exist two constants $C, c >0$ such that, for all $n \in \mathbb N$, for $1$-Lipschitz \emph{linear} mapping $u : E_n \rightarrow \mathbb{R}$ and for all $t>0$, 
     \begin{align*}
      \mathbb P \left(\left\vert u(Z_n - \tilde Z_n)\right\vert\geq t\right) \leq C e^{(t/c\sigma_n)^q}.
     \end{align*}
   We denote in that case $Z_n \in \tilde Z_n \pm \mathcal E_{q}(\sigma_n)$.
   \end{definition}
The deterministic equivalent around which the concentration occurs can be chosen indifferently in a ball of diameter having the same order as the observable diameter of $Z$.
\begin{lemma}\label{lem:deterministic_equivalent_sous_le _diametre_observable}
  Given a (sequence of) random vectors $Z$, two (sequences of) deterministic vector $\tilde Z_1,\tilde Z_2 \in \mathbb R$:
  \begin{align*}
    Z \in \tilde Z_1 \pm \mathcal E_q(\sigma) \ \ \text{  and  } \ \  |\tilde Z_1 - \tilde Z_2 | \leq O(\sigma)&
    &\Longrightarrow&
    &Z \in \tilde Z_2 \pm \mathcal E_q(\sigma).
  \end{align*}
\end{lemma}

\medskip

In the following, we will thus assume that $Z=(z_1,\ldots,z_n)$ is concentrated.
\begin{assumption}\label{ass:Concentration_Z}
  $Z \propto \mathcal E_2$.
\end{assumption}
As a $1$-Lipschitz projection of $Z \propto \mathcal E_2$, $z_i\propto \mathcal E_2$, and we can then conclude (see \cite{LOU19} for more details) that $\sup_{1\leq i\leq n}\|\Sigma_i\| \leq O(1)$. We also need to bound $\mu_i$ to control $E[z_iz_i^T] = \Sigma_i + \mu_i \mu_i^T$. 
\begin{assumption}\label{ass:borne_E_Z}
  $\sup_{1\leq i\leq n}\|\mu_i\| \leq O(1)$.
\end{assumption}

Given $\Delta \in \mathcal D_n^+$, the resolvent $\res^{Z}_\gamma(\Delta) = \res(Z\cdot Z^T,\Delta)= (\frac1n Z \Delta Z^T + \gamma I_p)^{-1}$ is a random matrix which exhibits useful properties to understand the statistics of $Z$ and more importantly its spectral behavior. In particular, the distribution of the singular values of $Z$ strongly relates to the well-known \textit{Stieltjes transform} $m_Z(z) = \frac{1}{p} \tr(\res_{-z}^Z(\Delta)$ where $z$ in a complex value distinct from any of the singular values of $Z$. This setting has already been extensively studied in \cite{LOU21RHL} where there was no diagonal matrix $\Delta$ but the column of $Z$ could have different distributions which leads to an equivalent setting.
The study is more simple here because we are just interested in the specific case where $z<0$ and for that reason, we further look at $\res \equiv \res_z^Z(\Delta)$ for $z>0$ and even $z\geq O(1)$ (for concentration issues).

It can be shown that $\res$ is a $\frac{2\|\Delta \|^{1/2}}{z^{3/2}\sqrt n}$-Lipschitz transformation of $Z$ and, therefore, assuming that $\frac{1}{z} \leq O(1)$, we can deduce that:
\begin{align}\label{eq:concentration_résolvante}
    \res \propto \mathcal E_2 \left(\frac{1}{\sqrt n} \right)
\end{align}
  A computable deterministic equivalent $\tilde Q$ of $Q$ is expressed thanks to a diagonal matrix $\Lambda^C(\Delta) \in \mathcal D_n^+$ defined, as the unique solution to:
  \begin{align*}
    \Lambda^C(\Delta) = I^C \left(\frac{\Delta}{I_n + \Delta \Lambda^C(\Delta)} \right);
  \end{align*}
  its existence and uniqueness is proven thanks to Corollary~\ref{cor:point_fixe_pour_tI} (here $f: \Delta \to \frac{\Delta}{I_n + \Delta \Lambda^C(\Delta)}$ is bounded by $\|\Delta\|$).
This equation allows us to compute $\Lambda^C(\Delta)$ iteratively via the standard fixed-point algorithm. The deterministic equivalent $\tilde \res_z^C(\Delta)$ of $\res_z^{ Z}(\Delta)$ is then easily computed and is defined as follows:
\begin{align*}
    \tilde \res_z^C(\Delta) \equiv \res_z \left( C,\frac{\Delta}{I_n+\Delta \Lambda^{ C}}\right)  = \left(\frac1n \sum_{i=1}^n \frac{\Delta_iC_i}{1+ \Delta_i\Lambda_i^C} + z I_p \right)^{-1}.
\end{align*}

\begin{theorem}[\cite{LOU21RHL}]\label{the:equivalent_deterministe_R}
  Given $\Delta \in \mathcal D_n^+$ and $A \in \mathcal M_p$ be deterministic matrices such that $\|\Delta\| \leq O(1)$ and\footnote{note here that the assumption on $A$ is lighter than in Theorem~\ref{the:Estimation_transforme_stieltjes}, mainly because here the diagonal matrix $\Delta$ is deterministic.} $\|A\|_F \equiv \sqrt{\tr(AA^T)} \leq O(1)$, we have the concentration:
  \begin{align*}
    \tr(A \res_z^{ Z}(\Delta)) \in \tr \left(A \tilde \res_z^C(\Delta) \right)\pm \mathcal E_2 \left(\frac{1}{\sqrt {n}}\right).
  \end{align*}
\end{theorem}
This theorem will later allow us to estimate the Stieltjes transform of the spectral distribution of $\hat C$ given at the beginning of the article. For this purpose, we need the next corollary to predict the asymptotic behavior of $\hat \Delta$ defined in \eqref{eq:point_fixe_Ch_Delta}. Recall that $I^X: \Delta\mapsto I(X\cdot X^T, \Delta)$. Then the following holds.
\begin{corollary}\label{cor:Concentration_I_m_D}
  For all $\Delta \in \mathcal D_n^+$ with $\| \Delta \| \leq O(1)$, 
  $$I^{ Z}(\Delta) \in \frac{\Lambda^{ C}(\Delta)}{I_n+\Delta\Lambda^{ C}(\Delta)} \pm \mathcal E_2\left(\frac{1}{\sqrt {n}}\right) \ \ \ \text{in} \ \ (\mathcal D_n, \|\cdot\|_F)$$
\end{corollary}
\color{black}
\begin{remark}\label{rem:extention_resultat_equivalent_deterministe_resolvante}
  It is possible to extend the results of Theorem~\ref{the:equivalent_deterministe_R} and Corollary~\ref{cor:Concentration_I_m_D} to the broader case where each mean $\mu_i$ ($i \in [n] $) can be decomposed as the sum of a particular component $\mathring \mu_i$ of low energy (i.e. with a low norm) and a bigger component proportional to a general signal $s$ of high energy as follows:
\begin{align}\label{eq:decomposition_moyenne}
  \mu_i = \mathring \mu_i + t_i s,
\end{align}
where $t_1,\ldots, t_n>0$ are $n$ scalars satisfying $\frac{1}{n}\sum_{i=1}^n t_i \geq 1$ and $\sup_{i\leq 1\leq n}t_i \leq O(1)$. The concentration results are then almost the same with a speed $O(\sqrt{\log n/n})$ replacing $O(1/\sqrt n)$ and the concentration in Theorem~\ref{the:equivalent_deterministe_R} is then only true for $A \in \mathcal{M}_{p}$ satisfying $\|A\|_* \leq O(1)$ (where $\|A\|_* = \tr((AA^T)^{1/2})$).
\end{remark}

\section{Estimation of the robust scatter matrix}\label{sec:robust_classification}
\subsection{Setting and strategy of the proof}
Having set up the necessary tools and preliminary results, we now concentrate on our target objective. Let $x_i= \sqrt{\tau_i} z_i + m$, $1\leq i\leq n$, where $\tau_i$ is a deterministic positive variable, $m\in \mathbb R^p$ is a deterministic vector, and $z_1,\ldots, z_n$ are the random vectors presented in the previous section. 
For $X=(x_1,\ldots,x_n) \in \mathcal M_{p,n}$, we write $X=  Z \tau^{\frac12} + m \un^T$ where $\tau \equiv \diag(\tau_i)_{1\leq i\leq n}\in \mathcal D_n^+$ and $\un \equiv (1,\ldots,1) \in \mathbb R^n$. The basic idea to estimate $\hat \Delta$, as a solution to the fixed-point equation $\hat \Delta = I^X(u(\hat \Delta))$, consists in retrieving a deterministic equivalent also solution to a (now deterministic) fixed-point equation. For this, we use the following central perturbation result.
\begin{theorem}\label{the:concentration_point_fixe_stable}
  Let $f,f'$ be two stable functions of $\mathcal D_n^+$, each admitting a fixed point $\Delta ,\Delta' \in \mathcal D_n^+$ as
  \begin{align*}
    \Delta = f(\Delta)&
    &\text{and}&
    &\Delta' = f'(\Delta').
   \end{align*}
   Further assume that $\Delta' \sim O(1)$ (i.e. that $\Delta \geq O(1)$ and $\Delta \leq O(1)$), that $f$ is contracting for the stable semi-metric around $\Delta'$ with a Lipschitz parameter $\lambda<1$, and that\footnote{In the application of the Theorem present in our paper, we are either in cases where $\|f(\Delta') - f'(\Delta')\| \underset{p,n \to \infty}{\to} 0$ for $\Delta$ and $\Delta'$ deterministic (Proposition~\ref{pro:deuxieme_equivalent_det}) or in cases where $\Delta$ is random but for any $K\geq 0$, with very high probability, $\|f(\Delta') - f'(\Delta')\| \leq K$ (Proposition~\ref{pro:estimation_gamma}). In both cases, we are thus left to verifying that $1-\lambda \geq O(1)$.} $$1-\lambda - \left\Vert \sqrt{\frac{f(\Delta') - f'(\Delta')}{\Delta'}}\right\Vert \geq O(1).$$
   Then, there exists a constant $K\leq O(1)$ such that
  \begin{align*}
    \left\Vert \Delta - \Delta' \right\Vert \leq K \| f(\Delta') - f'(\Delta') \|. 
  \end{align*}
\end{theorem} 
\begin{proof}
  Let us first bound:
  \begin{align*}
    \left\Vert \frac{\Delta - \Delta'}{\sqrt{\Delta f(\Delta')}}\right\Vert 
    &\leq d_s(f(\Delta),f(\Delta')) + \left\Vert \frac{f(\Delta') - \Delta'}{\sqrt{\Delta f(\Delta')}}\right\Vert
    \leq \lambda \left\Vert \frac{\Delta - \Delta'}{\sqrt{\Delta \Delta'}}\right\Vert + \left\Vert \frac{f(\Delta') - \Delta'}{\sqrt{\Delta f(\Delta')}}\right\Vert.
  \end{align*}
   (one must be careful here that the stable semi-metric does not satisfy the triangular inequality). Besides:
  \begin{align*}
    \left\Vert \frac{\Delta - \Delta'}{\sqrt{\Delta \Delta'}}\right\Vert 
    &\leq \left\Vert \frac{\Delta - \Delta'}{\sqrt{\Delta }} \left(\frac{\sqrt{f(\Delta')} -\sqrt{f'(\Delta')}}{\sqrt {\Delta'}\sqrt{f(\Delta')}}\right)\right\Vert  + \left\Vert \frac{\Delta - \Delta'}{\sqrt{\Delta f(\Delta')}}\right\Vert\\
    &\leq \left\Vert \frac{\Delta - \Delta'}{\sqrt{\Delta f(\Delta')}}\right\Vert \left(1 + \left\Vert \sqrt{\frac{f(\Delta') - f'(\Delta')}{\Delta'}}\right\Vert \right).
  \end{align*}
  Thus, by hypothesis, setting $K' = \frac{1}{1-\lambda -\varepsilon} \leq O(1)$, we have the inequality:
  \begin{align*}
    \left\Vert \frac{\Delta - \Delta'}{\sqrt{\Delta f(\Delta')}}\right\Vert 
    &\leq K'\left\Vert \frac{f(\Delta') - \Delta'}{\sqrt{\Delta f(\Delta')}}\right\Vert.
  \end{align*}
  Thus, as $O(1) \leq \|\Delta'\| -O(a_s) \leq f(\Delta')  \leq \|\Delta'\| +O(a_s) \leq O(1)$, we obtain the bound:
  \begin{align}\label{eq:phi_gamma_avec_denominateur}
    \left\Vert \frac{\Delta - \Delta'}{\sqrt{\Delta}}\right\Vert 
    &\leq K'\!{}'\left\Vert \frac{f(\Delta') - \Delta'}{\sqrt{\Delta}}\right\Vert
  \end{align}
  for some constant $K'\!{}'>0$.
  We are left to bound from below and above $\|\Delta \|$ to recover the result of the theorem from \eqref{eq:phi_gamma_avec_denominateur}. Considering the index $i_0$ such that $\Delta_{i_0} = \min(\Delta_i)_{1\leq i \leq n}$, we have:
  \begin{align*}
    \left\vert \Delta_{i_0} - \Delta'_{i_0}\right\vert \leq K'\!{}'\sqrt{\Delta_{i_0}} \left\Vert \frac{f(\Delta') - \Delta' }{\sqrt{\Delta_{i_0}}}\right\Vert  \leq O(a_s),  
  \end{align*}
  so that $\Delta_{i_0} \geq \Delta'_{i_0} -O(a_s) \geq O(1)$. On the other hand, one can bound again from \eqref{eq:phi_gamma_avec_denominateur}:
  \begin{align*}
    \|\sqrt{\Delta}\| \leq \left\Vert \frac{\Delta'}{\sqrt{\phi}}\right\Vert + K'\!{}'\left\Vert \frac{f(\Delta') - \Delta'}{\sqrt{\Delta}}\right\Vert \leq O(1).
  \end{align*}
  As a consequence, $\Delta \sim O(1)$, and we can conclude from \eqref{eq:phi_gamma_avec_denominateur}.
\end{proof}
Theorem~\ref{the:concentration_point_fixe_stable} can be employed when $\Delta$ is random and $\Delta'$ is a deterministic equivalent (yet to be defined). 
If we let $f = I^X \circ u(\cdot)$ (and thus $\Delta = \hat \Delta$), it is not possible to state that $\Delta \sim O(1) $ since $I^X \circ u(\cdot) = \diag(\frac{1}{n}x_i^T \res^X \circ u(\cdot) x_i)_{1\leq i \leq n}$ scales with $\tau$ which might be unbounded.
For this reason, in place of $\hat\Delta$, we will consider $\hat D \equiv \frac{\hat \Delta}{\underline \tau}$ where: 
\begin{align*}
    \underline \tau \equiv \diag(\max(\tau,1)) \geq I_n.
\end{align*}
We similarly denote $\bar \tau \equiv \diag(\min(\tau,1)) \leq I_n$, note that $\tau = \bar \tau \underline \tau$.

\subsection{Definition of $\tilde D$, the deterministic equivalent of $\hat D$}
The matrix $\hat D \equiv \frac{\hat \Delta}{\underline \tau}$ satisfies the fixed point equation
\begin{align*}
  \hat D =  I^{\bar Z}(u^\tau(\hat D)), &
  &\text{where}&
  &\bar z_i \equiv \frac{x_i}{\sqrt{\underline\tau_i}} = \sqrt {\bar \tau_i} z_i +\frac{ m}{\sqrt{\underline\tau_i}}&
  &\text{and}&
  &u^\tau : \Delta \mapsto \underline \tau u(\underline\tau \Delta).
\end{align*}
We will note from now on $\bar m_i = \mathbb E[\bar z_i]$ and $\bar C_i = \mathbb E[\bar z_i\bar z_i^T] $. In order to apply Corollary~\ref{cor:Concentration_I_m_D} with the hypothesis described in Remark~\ref{rem:extention_resultat_equivalent_deterministe_resolvante}, we will need a bound on the energy of the signal and on the $\tau_i$'s.
\begin{assumption}\label{ass:m_O_n}
  $\|m\|= O(1)$.
\end{assumption}

We still cannot apply Corollary~\ref{cor:Concentration_I_m_D} since $\|u^\tau(\hat D)\|$ is possibly unbounded. Still, let us assume for the moment that $\|u^\tau(\hat D)\|$ is indeed bounded: then, following our strategy, we are led to introducing a deterministic diagonal matrix $\tilde D$ ideally approaching $\hat D$ and satisfying
\begin{align}\label{eq:systeme_lambda_Deltat}
  \tilde D = \frac{\Lambda^{\bar C}(u^{ \tau}(\tilde D))}{I_n+u^{ \tau}(\tilde D)\Lambda^{\bar C}(u^{ \tau}(\tilde D))},
\end{align}
(where we recall $\Lambda^{\bar C}(u^{ \tau}(\tilde D)) = \check I(\bar C,\frac{u^{ \tau}(\tilde D)}{I_n+ u^{ \tau}(\tilde D)\Lambda^{\bar C}(u^{ \tau}(\tilde D))})$). 
Before proving the validity of the estimate $\tilde D$ of $\hat D$, let us justify the validity of its definition (i.e., the existence and uniqueness of $\tilde D$). To this end, we first introduce a stable auxiliary mapping $\eta : \mathbb R_+ \to \mathbb R_+$.
\begin{proposition}\label{pro:definition_eta}
  Let $x \in \mathbb R^+$. Then the equation
  \begin{align*}
    \eta = \frac{1}{\frac{1}{x} + u(\eta)}, \eta \in \mathbb R^+.
  \end{align*}
  admits a unique solution that we denote $\eta(x)$. The mapping $\eta: \mathbb R \rightarrow \mathbb R$ is stable.
\end{proposition}
\begin{proof}
  Let us show  that the mapping $f : \eta \mapsto x/(1+x u(\eta))$ is contracting for the stable semi-metric:
  \begin{align*}
    d_s(f(\eta), f(\eta'))
    &=d_s \left( \frac{1}{f(\eta)}, \frac{1}{f(\eta')} \right)= \frac{\left\vert u(\eta)-u(\eta')\right\vert}{\sqrt{\left(\frac{1}{x} + u(\eta)\right) \left(\frac{1}{x} + {u(\eta')}\right)}} \\
    & \leq \sqrt{\frac{{u(\eta)u(\eta')}}{\left(\frac{1}{x} + {u(\eta)}\right) \left(\frac{1}{x} + {u(\eta')}\right)}} d_s(u(\eta), u(\eta'))\\
    & \leq \frac{1}{1+\frac{1}{u^\infty x}}d_s(\eta, \eta') \ \ 
    \leq  u^\infty_\times d_s(\eta, \eta')
  \end{align*}
  The Theorem~\ref{the:point_fixe_fonction_contractante_de_D_n_borne_superieurement} then allow us to conclude on the existence and uniqueness of $\eta(x)$.
  To prove the stability of $\eta$, we are going to use the characterization with the monotonicity of the functions $\eta_/: x \mapsto \frac{\eta(x)}{x}$ and $\eta_\times \mapsto x \eta(x)$ presented in Property~\ref{prot:Caracterisation_stabilite_monotonie}. Let us consider $x,y \in \mathbb R^+$ such that $x\leq y$; if $\eta(x) \leq \eta (y)$, then $\eta_\times(x) \leq \eta_\times(y)$. Besides, since in addition $u_/$ is non-decreasing,
  \begin{align*}
    \eta_/(x) = \frac{1}{1+xu(\eta(x))} 
    \geq \frac{1}{1+\frac{y\eta(y)u(\eta(x))}{\eta(x)}}
    \geq \frac{1}{1+yu(\eta(y))} = \eta_/(y).
  \end{align*}
  Similarly, if $\eta(x) \geq \eta (y)$, then $\eta_/(x) \geq \eta_/(y)$ and
  \begin{align*}
    \eta_\times(x) = \frac{1}{\frac1{x^2} + \frac{u(\eta(x))}{x}} \leq \frac{1}{\frac1{y^2} + \frac{\eta(x)}{x}\frac{u(\eta(y))}{\eta(y)}}\leq \frac{1}{\frac1{y^2} + \frac{u(\eta(y))}{y}} = \eta_ \cdot(y).
  \end{align*}
  We see that in both cases $\eta_/(x) \geq \eta_/(y)$ and $\eta_\times(x) \leq \eta_\times(y)$. Therefore, thanks to Property~\ref{prot:Caracterisation_stabilite_monotonie}, $\eta \in \mathcal S(\mathbb R^+)$.
\end{proof}
The first equation of \eqref{eq:systeme_lambda_Deltat} can be rewritten 
$\tilde D = \eta_\tau(\Lambda^{\bar C}(u^\tau(\tilde D)))$, with $\eta_\tau : x \mapsto \frac{\eta(\underline \tau x)}{\underline \tau}$. To define $\tilde D$ properly, we thus need to show that $\Lambda^{\bar C}$ is stable (with the aim of employing Theorem~\ref{the:point_fixe_fonction_contractante_de_D_n_borne_superieurement} again).
\begin{proposition}\label{pro:Lambda_stable}
  For any $S \in \mathcal S_p^n$, the mapping $\Lambda^{S}: \mathcal D_n^+ \to   \mathcal D_n^+$ is stable and satisfies $\forall \Delta, \Delta' \in \mathcal D_n^+$:
  \begin{align*}
    d_s \left(  \Lambda^S(\Delta), \Lambda^S(\Delta') \right) \leq \max(\phi^S_\gamma(\Delta) ,\phi^S_\gamma(\Delta') ) d_s(\Delta, \Delta')
  \end{align*}
  (with the notation provided in Lemma~\ref{lem:I_conctractante}.)
\end{proposition}
\begin{proof}
  Given $S \in \mathcal S_p^n$ and $\Delta, \Delta' \in \mathcal D_n^+$, there exists $i_0 \in [n]$ such that, if we note $\lambda \equiv \max(\phi^S_\gamma(\Delta) b,\phi^S_\gamma(\Delta') )$, we can bound thanks to Lemma~\ref{lem:I_conctractante}:
  \begin{align*}
    d_s(\Lambda^S(\Delta),\Lambda^S(\Delta') ) 
    &= d_s \left(\check I \left( S,\frac{\Delta}{I_n+ \Delta\Lambda^{ S}(\Delta)}\right),\check I \left( S,\frac{\Delta'}{I_n+ \Delta'\Lambda^{ S}(\Delta')}\right)  \right) \\
    &\leq \lambda d_s \left(\frac{\Delta}{I_n+ \Delta\Lambda^{ S}(\Delta)},\frac{\Delta'}{I_n+ \Delta'\Lambda^{ S}(\Delta')}  \right)\\
    &= \lambda d_s \left(\frac{I_n}{\Delta} + \Lambda^{ S}(\Delta),\frac{I_n}{\Delta'} + \Lambda^{ S}(\Delta')  \right)\\
    &=\lambda\frac{\left\vert \frac{1}{\Delta_{i_0}} + \Lambda^{ S}(\Delta)_{i_0} - \frac{1}{\Delta'_{i_0}} + \Lambda^{ S}(\Delta')_{i_0}\right\vert}{\sqrt{ \left(\frac{1}{\Delta_{i_0}} + \Lambda^{ S}(\Delta)_{i_0}\right) \left(\frac{1}{\Delta'_{i_0}} + \Lambda^{ S}(\Delta')_{i_0} \right) }}\\
    &\leq\lambda \max \left(\frac{\left\vert \frac{1}{\Delta_{i_0}}  - \frac{1}{\Delta'_{i_0}} \right\vert}{\sqrt{ \frac{1}{\Delta_{i_0}} \frac{1}{\Delta'_{i_0}}}}, \frac{\left\vert \Lambda^{ S}(\Delta)_{i_0} - \Lambda^{ S}(\Delta')_{i_0}\right\vert}{\sqrt{  \Lambda^{ S}(\Delta)_{i_0} \Lambda^{ S}(\Delta')_{i_0}  }}\right)
    \\
    &\leq \lambda\max \left(d_s(\Delta,\Delta'),d_s(\Lambda^{ S}(\Delta), \Lambda^{ S}(\Delta'))\right)
  \end{align*}
  Thanks to Lemma~\ref{lem:I_conctractante}, the stability rules given in Property~\ref{prot:stabilite_contr_mapping_mat_diagonales}, and the extra tools given by Lemma~\ref{lem:stability_preliminary_lemma} (already used to prove Property~\ref{prot:stabilite_contr_mapping_mat_diagonales}). As a conclusion,
  \begin{align*}
    d_s(\Lambda^S(\Delta),\Lambda^S(\Delta') )  \leq \lambda \max \left(d_s(\Delta,\Delta'),d_s(\Lambda^{ S}(\Delta), \Lambda^{ S}(\Delta'))\right),
  \end{align*}
  which directly implies that $d_s(\Lambda^S(\Delta),\Lambda^S(\Delta') )  \leq \lambda d_s(\Delta,\Delta')$. 
\end{proof}
We are thus now allowed to define $\tilde D$.
\begin{proposition}\label{pro:Definition_equivalent_deterministe}
  There exists a unique diagonal matrix $\tilde D \in \mathcal D_n^+$ satisfying \eqref{eq:systeme_lambda_Deltat}.
\end{proposition}
\begin{proof}
    We already know from Proposition~\ref{pro:Lambda_stable} that $D \mapsto \Lambda^{\bar C}(u^\tau(D))$ is contractive for the semi metric $d_s$ (we can indeed show as in the proof of Corollary~\ref{cor:point_fixe_pour_tI} that $\sup_{\Delta \in \mathcal D_n^+}\Phi^{\bar C}_\gamma(u^\tau(\Delta)) < 1$ thanks to Lemma~\ref{lem:tR_borne} and since $u^\tau \leq \|\underline \tau \| u^\infty$ -- be careful here that, possibly $\|\underline \tau \| \geq O(n)$, but that is not the question here). The same is true for $\eta_\tau(\Lambda^{\bar C}(u^\tau(\tilde D)))$ since $\eta$ is stable; 
     the existence and uniqueness of $\tilde D$ thus unfold from Theorem~\ref{the:point_fixe_fonction_contractante_de_D_n_borne_superieurement}.
\end{proof}

\subsection{Concentration of $\hat D$ around $\tilde D$}

 In order to establish the concentration of $\hat D$, we need an assumption on $\eta$ to be able to bound $\tilde D = \eta_\tau(\Lambda^{\bar C}(u^\tau(\tilde D)))$. This assumption is expressed through a condition on $u$, justified by the following lemma that we already made visible on Figure~\ref{fig:exemple_u}.
\begin{lemma}\label{lem:eta_/_borne}
  The mapping $\eta_/$ is bounded from below iff,  $\forall t\in \mathbb R^+$, $u_\times(t) = tu(t) < 1$.
\end{lemma}
\begin{proof}
  If there exists $\alpha >0$ (and $\alpha<1$) such that $\forall x\in \mathbb R^+$, $\frac{\eta(x)}{x} \geq \alpha$, then
  \begin{align*}
    \frac{\eta(x)}{x} + (1-\alpha) \geq 1&
    &\text{and therefore:}&
    &\frac{1}{\frac{1}{x} + u(\eta(x))} = \eta(x) \geq \frac{1}{\frac{1}{x} + \frac{1-\alpha}{\eta(x)}},
  \end{align*}
  which implies that $u(\eta(x)) \eta(x) \leq 1-\alpha$. But since $\eta$ is not bounded (otherwise $\lim_{t \to \infty} \frac{\eta(t)}{t} = 0 <\alpha$), there exists a sequence $(x_n)_{n\geq 0} \in \mathbb R_+^{S}$ such that $\eta(x_n) \to \infty$. Thus ($u_\times$ being non-decreasing), $\forall t >0, u_\times (t) \leq \lim_{n \to \infty}u(\eta(x_n)) \eta(x_n)  \leq 1-\alpha$.
  Conversely, if $\forall t>0, u_\times(t) <1$, $\forall x\in \mathbb R^+$:
  \begin{align*}
    \frac{\eta(x)}{x} \geq \frac{1}{1 + u_\times^\infty \frac{x}{\eta(x)}}&
    &\text{thus}&
    & \frac{\eta(x)}{x} \geq 1- u_\times^\infty > 0.
  \end{align*}
\end{proof}
\begin{assumption}\label{ass:u_time_borne}
   $u_\times^\infty< 1$, where $u_\times^\infty  = \lim_{t \to \infty} tu(t)$.
\end{assumption}

We complete this extra assumption with  two rather ``loose'' assumptions on $\tau$, and on the non-centered covariance matrices $C_i$.
\begin{assumption}\label{ass:tr_C_i_borne_below}
  $\inf_{1\leq i\leq n}\frac{1}{n}\tr C_i \geq O(1)$.
\end{assumption}
\begin{assumption}\label{ass:tr_C_i_borne_below}
  $\frac{1}{n}\sum_{i = 1}^n \tau_i \leq O(1)$.
\end{assumption}

These assumptions imply the following important control.
\begin{lemma}\label{lem:borne_tgamma}
  $\tilde D \sim O(1)$.
\end{lemma}
\begin{proof}
  We already know from our assumptions that $O(1)\leq \frac{1}{n}\tr(C_i) + \frac{1}{n}\frac{m^Tm}{\underline \tau} = \frac{1}{n}\tr \bar C_i \leq O(1)$ and we can then bound:
  \begin{align*}
    \Lambda^{\bar C}(u^\tau(\tilde D))_i 
    \geq \frac{O(1)}{\gamma + \frac{1}{n}\|\sum C_i \tau_i u(\tau_i \tilde D) \|} 
    \geq \frac{O(1)}{\gamma + \sup_{i\in [n]}\|C_i \| \frac{1}{n}\sum_{i= 1}^n \tau_i  u^\infty} \geq O(1).
  \end{align*}
  Therefore, we can conclude thanks to Assumption~\ref{ass:u_time_borne} and Lemma~\ref{lem:eta_/_borne}:
  \begin{align*}
    \tilde D = \frac{1}{\tau_i} \eta \left( \tau_i \Lambda^{\bar C}(u^\tau(\tilde D))_i \right) \geq \eta_/^\infty \Lambda^{\bar C}(u^\tau(\tilde D))_i \geq O(1).
   \end{align*} 
\end{proof}

This control allows us to establish the concentration of $\hat D$:
\begin{proposition}\label{pro:estimation_gamma}
    There exist two constants $C,c>0$ ($C,c \sim O(1)$) such that: 
    \begin{align*}
       \forall \varepsilon\in (0,1]:&
       &\mathbb P \left(\left\Vert \hat D - \tilde D\right\Vert \geq \varepsilon\right) \leq C e^{-cn\varepsilon^2/\log n}.
     \end{align*} 
\end{proposition}
\begin{proof}
  Let us check the hypotheses of Theorem~\ref{the:concentration_point_fixe_stable}. Let us first bound the Lipschitz parameter (for the stable semi-metric) $\lambda$ of $  I^{\bar Z}\circ u^\tau$ around $\tilde D$ defined as:
\begin{align*}
      \forall \Delta \in \mathcal D_n^+: \ \ \left\Vert \frac{  I^{\bar Z}(u^\tau(\Delta)) -   I^{\bar Z}(u^\tau(\tilde D))}{ \sqrt{  I^{\bar Z}(u^\tau(\Delta))   I^{\bar Z}(u^\tau(\tilde D))}}\right\Vert < \lambda \left\Vert\frac{\Delta -\tilde D}{ \sqrt{\Delta \tilde D}} \right\Vert.
  \end{align*}
  An inequality similar as in Lemma~\ref{lem:I_conctractante} gives us:
  \begin{align*}
      \lambda \leq \sqrt{\Vert 1- \gamma  \res^{\bar Z}(u^\tau(\tilde D)\Vert}\leq 1 - \frac{\gamma}{\gamma+ \frac{1}n \|u^\tau(\tilde D) \| \|\bar Z\bar Z^T \|}
  \end{align*}
  (thanks to Lemma~\ref{lem:tR_borne}). 

  First, the concentration $\bar Z \propto \mathcal E_2$ implies the concentration of its norm $\|Z\| \in \mathbb E[\|Z\|] \pm \mathcal E_2$ satisfying $K \equiv \mathbb E[\|Z\|]/\sqrt n \leq O(1)$. There exist then two supplementary constants $C,c>0$ such that for all $n \in \mathbb N$, $\forall t>0$:
  \begin{align*}
       \mathbb P \left( \frac{\|\bar Z\|}{\sqrt n} \geq t + K \right) \leq Ce^{-nt^2},
     \end{align*}
  and with, say, $t =K$, we see that with probability larger than $1 - Ce^{-cK^2n}$ (for some constants $C,c>0$), $\|{\bar Z}\| \leq 2K \sqrt n$.
  There exists then a constant $K'>0$, such that under this highly probable event $1-\lambda \geq K'$. 
    
  Second, we know that $\bar Z = Z \bar\tau^{1/2} \sim \mathcal E_2$ and $u^\tau(\tilde D)\leq \frac{u_\times ^\infty}{\tilde D} \leq O(1)$ from Proposition~\ref{lem:borne_tgamma}, therefore, ${\bar Z}u^\tau(\tilde D) \propto \mathcal E_2$ and we can then employ
   Corollary~\ref{cor:Concentration_I_m_D} to state that $  I^{\bar Z}(u^\tau(\tilde D)) \in \tilde D \pm \mathcal E_2(1/\sqrt{n})$. Thus there exist two constants $C,c>0$ such that
  \begin{align*}
    \forall t>0: \ \ \pi_t \equiv \mathbb P \left(\left\Vert   I^{\bar Z}(u^\tau(\tilde D)) - \tilde D \right\Vert \geq t \right)
    &\leq \sum_{i\in \mathbb N}\mathbb P \left(\left\vert   I^{\bar Z}(u^\tau(\tilde D))_i - \tilde D_i \right\vert \geq t \right) 
    \leq n C e^{-cnt^2}.
  \end{align*}
  Now since $\pi_t\leq 1$, we can choose two constants $c', C'>0$ 
  depending on $c,C$ but independent with $n$ such that $\pi_t \leq C'e^{-c'nt^2/\log n}$.

  In this last inequality, we can take $t$ small enough ($t = \frac{K'{}^2}{4\|1/\tilde D\|}$) such that on an event of probability larger than $1 - C' e^{-c'\!{}'n/\log n}$ ($c'\!{}'>0$), we have:
  \begin{align*}
    1 - \lambda - \sqrt{\left\Vert \frac{  I^{\bar Z}(u^\tau(\tilde D)) - \tilde D}{\tilde D}\right\Vert} \geq \frac{K'}{2}.
  \end{align*}
  Let us introduce the event:
  \begin{align*}
     \mathcal A \equiv \left\{ \|{\bar Z}\| \leq 2K \sqrt n \ \ \text{and} \ \  \left\Vert   I^{\bar Z}(u^\tau(\tilde D)) - \tilde D \right\Vert \geq \frac{K'{}^2}{4\|1/\tilde D\|}\right\}.
   \end{align*} 
  It satisfies $\mathbb P(\mathcal A^c) \leq C'\!{}' e^{-c'\!{}'n}$, for some constants $C'\!{}',c'\!{}'>0$.

  Applying Theorem~\ref{the:concentration_point_fixe_stable}, we know that there exists a constant $K \leq O(1)$ such that $\forall n \in \mathbb N$:
  \begin{align*}
    \forall t>0: \ 
     \mathbb P \left(\left\Vert   \hat D - \tilde D \right\Vert \geq t \right)
     &= \mathbb P \left(\left\Vert   \hat D - \tilde D \right\Vert \geq t, \mathcal A\right) + C' e^{-c'n} \\
     & \leq \pi_{t/K} + C'\!{}' e^{-c'\!{}'n} \ \ \leq \ C' e^{-c'nt^2/ \log n} + C'\!{}' e^{-c'\!{}'n}.
   \end{align*} 
   We thus retrieve the result of the proposition bounding the value of $t$ and choosing $C$ and $c$ appropriately.
\end{proof}
It is even possible to give a deterministic equivalent of $\hat D$ independent of the signal $m$.
\begin{proposition}\label{pro:deuxieme_equivalent_det}
  The fixed-point equation $D = \eta_\tau \circ \Lambda^{\bar \tau C} \circ u^\tau (D)$ admits a unique solution, denoted $\tilde D_{-m} \in \mathcal D_n^+$, and which satisfies $\|\tilde D - \tilde D_{-m} \| \leq O \left(\frac{1}{\sqrt n}\right)$.
\end{proposition}
\begin{proof}
Let us first recall the notations $\mu_i = \mathbb E[z_i]$, $C_i = \mathbb E[z_i^t] = \Sigma_i + \mu_i \mu_i^T$ and: 
\begin{align*}
  \bar C_i = \bar \tau_i \Sigma_i + \left( \sqrt{\bar \tau_i}\mu_i + \frac{m}{\sqrt{\underline \tau_i}} \right)\left( \sqrt{\bar \tau_i}\mu_i + \frac{m}{\sqrt{\underline \tau_i}} \right)^T = \bar \tau_i \Sigma_i  + \bar m_i \bar m_i^T.
\end{align*}
  The existence and uniqueness of $\tilde D_{-m}$ are justified for the same reasons as for $\tilde D$ (just take $m=0$). We want to employ again Theorem~\ref{the:concentration_point_fixe_stable}, with the deterministic mappings:
  \begin{align*}
      f =  \eta_\tau \circ \Lambda^{\bar \tau C} \circ u^\tau&
      &\text{and}&
      &f' = \eta_\tau \circ \Lambda^{\bar C} \circ u^\tau,
  \end{align*}
  and with $\Delta = \tilde D_{-m}$ and $\Delta' = \tilde D$. We note that $\tilde D \sim O(1)$ and the Lipschitz parameter $\lambda$ of $f$ for the semi-metric satisfies a similar inequality as in the proof of Proposition~\ref{pro:estimation_gamma}:
  $$1 - \lambda \geq \frac{\gamma}{\gamma+u_\times^\infty \|\frac1{\tilde D}\|\sup \|C_i\|}\geq O(1).$$
  We then need to bound the spectral norm $\Vert \eta_\tau \circ \Lambda^{\bar \tau C} \circ u^\tau (\tilde D) - \eta_\tau \circ \Lambda^{\bar C} \circ u^\tau (\tilde D)\Vert$. 
  Note that $\eta$ is $1$-Lipschitz for the absolute value because, for any $x,y \in \mathbb R^+$, the stability of $\eta$ implies: 
  \begin{align*}
    \frac{|\eta(x)-\eta(y)|}{|x-y|} \leq \sqrt{\frac{\eta(x)\eta(y)}{xy}} = \sqrt{\frac{1}{\left(1 +xu(\eta(x)\right)\left(1 +yu(\eta(y))\right)}} \leq 1.
  \end{align*}
  Thus $\eta_\tau$ is also $1$-Lipschitz. We are then left to bounding the distance (in spectral norm) between $\Lambda^{\bar \tau C} \circ u^\tau (\tilde D)$ and $\Lambda^{\bar C} \circ u^\tau (\tilde D)$, and we are naturally led to employing a second time Theorem~\ref{the:concentration_point_fixe_stable} since those two values are both fixed points of stable mappings:
  \begin{align*}
    \Lambda^{\bar C}(u^\tau (\tilde D)) = \tilde I_{u^\tau (\tilde D)}^{\bar C}(\Lambda^{\bar C}(u^\tau (\tilde D)))&
    &\text{and}&
    &\Lambda^{\bar \tau C}(u^\tau (\tilde D)) = \tilde I_{u^\tau (\tilde D)}^C(\Lambda^{\bar \tau C}(u^\tau (\tilde D)))
   \end{align*}
   where, for any $S \in \mathcal S_p^n$ and $\Delta \in \mathcal D_n^+$, $\tilde I^S_\Delta : \Lambda \mapsto  I \left(S,\frac{\Delta}{I_n + \Delta\Lambda}  \right)$. Once again, the first hypothesis is satisfied, $\Lambda^{ C}(u^\tau (\tilde D)) \sim O(1)$ and $\lambda'$, the Lipschitz parameter of $\tilde I^{\bar C}_\Delta$ satisfies $1-\lambda' \geq O(1)$. Noting for simplicity $\Delta \equiv u^\tau (\tilde D)$, $\Lambda \equiv \Lambda^{\bar \tau C}(\Delta)$ and $\tilde \res^S = \tilde \res^{S}(S, \frac{\Delta}{I_n + \Delta\Lambda})$ (for $S = \bar C$ or $S = C$), we are left to bounding, for any $i \in [n]$,
   \begin{align*}
    \left\vert \tilde I_{\Delta}^{\bar \tau C}(\Lambda)_i - \tilde I_{\Delta}^{\bar C}(\Lambda)_i\right\vert
    &\leq \frac{1}{n\underline\tau_i}m^T \tilde \res^{\bar\tau C} m + \frac{2}{n}\sqrt{\frac{\bar\tau_i}{\underline\tau_i}}m_i^T \tilde \res^{\bar \tau C} m \\
    &\hspace{0.2cm}+ \left\vert \frac{1}{n} \tr \left(C_i \tilde \res^{\bar \tau C} \left(\frac{1}{n}\sum_{j=1}^n \sqrt{\frac{\bar\tau_j}{\underline\tau_j}}(m_j^Tm+mm_j^T) + \frac{1}{\underline \tau_j} m m^T\right)\tilde \res^{\bar C} \right)\right\vert \\
    &\leq O \left(\frac{1}{ n} + \frac{1}{n} m^T \tilde \res^{\bar C} C_i \tilde \res^{\bar \tau C} m+ \frac{1}{n} \sup_{j\in [n]}m^T \tilde \res^{\bar C} C_i \tilde \res^{\bar \tau C} m_j\right)  \leq O \left(\frac{1}{ n}\right)
   \end{align*}
  since $\frac{1}{\underline\tau_j}, \sqrt{\frac{\bar\tau_j}{\underline\tau_j}} \leq 1$ $\forall j \in [n]$.
  Applying twice Theorem~\ref{the:concentration_point_fixe_stable}, we retrieve the result of the proposition.
\end{proof}

Propositions~\ref{pro:estimation_gamma} and~\ref{pro:deuxieme_equivalent_det} allow us to set the following result, proved similarly as Lemma~\ref{lem:deterministic_equivalent_sous_le _diametre_observable}:
\begin{corollary}\label{cor:D_hat_m_D_m_m}
    There exist two constants $C,c>0$ ($C,c \sim O(1)$) such that: 
    \begin{align*}
       \forall \varepsilon\in (0,1]:&
       &\mathbb P \left(\left\Vert \hat D - \tilde D_{-m}\right\Vert \geq \varepsilon\right) \leq C e^{-cn\varepsilon^2/\log n}.
     \end{align*} 
\end{corollary}
We now have all the elements to prove Theorem~\ref{the:Estimation_transforme_stieltjes}.
\begin{proof}[Proof of Theorem~\ref{the:Estimation_transforme_stieltjes}]
  Let us set $U \equiv  \bar \tau u^\tau(\tilde D_{-m})$. We already know that $\|U\| \leq \sup_{i\in [n]}\frac{\bar \tau_iu_\times^\infty}{[\tilde D_{-m}]_i} \leq O(1)$, which allows us to set, on the one hand, that for any deterministic matrix $A \in \mathcal{M}_{p}$ satisfying $\|A\|_*\leq O(1)$:
  \begin{align*}
    \tr(A \res_z^{ Z}(U)) \in \tr \left(A \tilde \res_z^C(U) \right)\pm \mathcal E_2 \left(\frac{1}{\sqrt {n}}\right),
  \end{align*}
  thanks to Theorem~\ref{the:equivalent_deterministe_R} (it is true for all $A \in \mathcal{M}_{p}$ such that $\|A\|_F\leq O(1)$, so in particular for the matrices $A$ satisfying $\|A\|_* \leq O(1)$).

  On the second hand, placing ourselves in the overwhelming event where $\|Z\|\leq K \sqrt n$ as in the proof of Proposition~\ref{pro:estimation_gamma}, we can bound thanks to Propositions~\ref{pro:estimation_gamma} and~\ref{pro:deuxieme_equivalent_det}:
  \begin{align*}
    \left\vert \tr(A \res_z^{ Z}(U)) - \tr(A \res_z) \right\vert
    &\leq \frac{1}{n}\left\Vert  \res_z^{ Z}(U) Z \left( \bar \tau u^\tau(\hat \Delta)- \bar \tau u^\tau(\tilde \Delta_{-m})\right)  Z^T\res_z\right\Vert\\
    &\leq O \left( \sup_{i\in [n]} \frac{\bar \tau_i u_\times^\infty \Vert \hat D - \tilde \Delta_{-m} \Vert}{\hat D_i[\tilde \Delta_{-m}]_i} \right) \ 
  \end{align*}
  we can then conclude thanks to Corollary~\ref{cor:D_hat_m_D_m_m}.
\end{proof}
\section{Conclusion}

In this article, we have developed an original framework to study the large dimensional behavior of a family of matrices solution to a fixed-point equation, under a quite generic probabilistic data model (which notably does not enforce independence in the data entries). Recalling that most state-of-the-art statistical (machine) learning algorithms are optimization problems, having implicit solutions, which are then applied to complex data models, this work opens the path to a more systematic exploitation of concentration of measure theory for the large dimensional analysis of possibly complex machine learning algorithms and data models.

\newpage

\bibliographystyle{elsarticle-harv}
\bibliography{biblio}

\newpage
\appendix

\section{Assumptions}
\setcounter{assumption}{0}

We recollect here all the assumptions introduced in the core of the article.

\begin{assumption}
  $u\in \mathcal S(\mathbb R^+)$, $\exists u^\infty>0$ such that $\forall t \in \mathbb R^+,  \ u(t) \leq u^\infty$.
\end{assumption}
\begin{assumption}
  The random vectors $z_1,\ldots,z_n$ are all independents.
\end{assumption}
\begin{assumption}
  $p \leq O(n)$
\end{assumption}
\begin{assumption}
  $Z \propto \mathcal E_2$.
\end{assumption}
\begin{assumption}
  $\sup_{i\in[n]}\|\mu_i\| \leq O(1)$.
\end{assumption}
 \begin{assumption}
    $\|m\| \leq O(1)$.
  \end{assumption}
\begin{assumption}
  $u_\times^\infty<1$.
\end{assumption}
\begin{assumption}
  $\inf_{1\leq i\leq n}\frac{1}{n}\tr C_i \geq O(1)$.
\end{assumption}
\begin{assumption}
  $\frac{1}{n}\sum_{i = 1}^n \tau_i \leq O(1)$.
\end{assumption}

\section{Supplementary inferences on the stable semi-metric and topological properties}\label{app:stable_metic}
\begin{remark}\label{rem:stable_non_continu_en_zero}
  Not all the stable mappings admit a continuous continuation on $\overline{\mathcal D}_n^+$. To construct a counter example, for any $n \in \mathbb N$, let us note 
  \begin{itemize}
    \item $e_n: x \mapsto \frac{3}{2} - 2^{n-1}x$ (it satisfies $e_n(1/2^n) = 1$ and $e_n(3/2^{n+1}) = \frac{3}{4}$),
    \item $d_n: x \mapsto 2^{n-1}x$ (it satisfies $d_n(1/2^{n-1}) = 1$ and $e_n(3/2^{n+1}) = \frac{3}{4}$),
    \item $v_n : \mathbb R^+ \to \mathbb R^+$ satisfying for all $x \in \mathbb R^+$, $v_n(x) = \max(e_n,d_n)$ (in particular, $v_n(2^n) = v_n(2^{n-1}) = 1$),
    \item $f : \mathbb R^+ \to \mathbb R^+$ satisfying for all $x \in \mathbb R^+$, $f(x) = \inf_{n\in \mathbb N} v_n(x)$.
  \end{itemize} 
  We know from Property~\ref{prot:Caracterisation_stabilite_monotonie} that for all $n \in \mathbb N$, $d_n$ is stable and that $e_n$ is stable on $[0,3/2^{n+1}]$ (where $x \mapsto xe_n(x)$ is non decreasing) which eventually allows us to set that $f$ is stable, thanks to Corollary~\ref{cor:inf_sup_stable}.

  However $f$ does not admit continuous continuation on $0$ since:
  \begin{align*}
    \underset{n \to \infty} \lim f \left( \frac{1}{2^n} \right) = 1 \neq \frac{3}{4}= \underset{n \to \infty} \lim f \left( \frac{3}{2^n}\right).
  \end{align*}
\end{remark}

\begin{proof}[Proof of Proposition~\ref{pro:palliation_absence_inegalite_triangulaire}]
   For a given integer $p \geq 1$, let us differentiate the mapping:
   \begin{align*}
     f_p : \begin{aligned}[t]
       \mathbb R_+^{p-1}\hspace{0.4cm}& \longrightarrow&\mathbb R \hspace{1.8cm}\\
       (y_1,\ldots, y_{p-1})&\longmapsto&\frac{y_1-x}{\sqrt{y_1x}} + \cdots  \frac{z-y_{p-1}}{\sqrt{zy_{p-1}}}
     \end{aligned}
   \end{align*}
   one can compute for any $y_1,\ldots, y_{p-1} \in \mathbb R^+$ and $i \in [p-1]$:
   \begin{align}\label{eq:differentielle_f_p_triangulaire}
     \frac{\partial f_p(y_1,\ldots, y_{p-1})}{\partial y_i} = \frac{1}{2} \frac{1}{\sqrt{y_i y_{i-1}}} \left(1+ \frac{y_{i-1}}{y_i}\right) -\frac{1}{2} \frac{1}{\sqrt{y_{i+1} y_{i}}} \left(1+ \frac{y_{i+1}}{y_{i}}\right) 
   \end{align}
   (where $y_0$ and $y_{p}$ designate respectively $x$ and $z$)
   In particular, when $p=1$, for any $y\geq x >0$:
   \begin{align*}
     \frac{\partial }{\partial y} \left(\frac{y-x}{\sqrt{yx}}\right) = \frac{1}{2} \frac{1}{\sqrt{xy}} \left(1+\frac{x}{y}\right) \geq 0
   \end{align*}
   which proves the first result of the proposition. 
   Now if we assume that $y \leq x \leq z$:
   \begin{align*}
     d_s(x,y) + d_s(y,z) \geq d_s(x,z),
   \end{align*}
   and the same inequality holds if one assumes that $x\leq z \leq y$. Returning to the setting of the proposition, we can therefore place ourselves in the open space:
   \begin{align*}
     \mathcal U^p_{x,z} = \{(y_1,\ldots, y_p) \in \mathbb R^{p-1}_+,x < y_1 < \cdots < y_{p-1} < z\}.
   \end{align*}

   If one fixes $x,z \in \mathbb R^+$, then $f_p(x,y_1,\ldots, y_{p-1},z) = d_s(x,y_1) + \cdots + d_s(y_{p-1} ,z)$ is minimum for $y_1,\ldots, y_{p-1}$ satisfying:
   \begin{align*}
     \frac{1}{\sqrt{y_i y_{i-1}}} \left(1- \frac{y_{i-1}}{y_i}\right) = \frac{1}{\sqrt{y_{i+1} y_{i}}} \left(1- \frac{y_{i}}{y_{i+1}}\right)
   \end{align*}
   which is equivalent to $y_i = \sqrt{y_{i-1} y_{i+1}}$. Noting $\tilde x = \log(x),\tilde y_1 = \log({y_1}), \ldots, \tilde y_n = \log({y_n}), \tilde z = \log(z) $, we see that this identity writes $\tilde y_i = \frac{1}{2} (\tilde y_{i-1} \tilde y_{i+1})$, which implies $\tilde y_i = \tilde x + \frac{i}{p} (\tilde z - \tilde x)$, or in other words:
   \begin{align*}
     y_i = x^{\frac{p-i}{p}} z^{\frac{i}{p}}.
   \end{align*}
   In that case:
   \begin{align*}
     d_s(y_i, y_{i+1}) 
     = \left\vert \frac{x^{\frac{p-i}{2p}} z^{\frac{i}{2p}}}{x^{\frac{p-i -1}{2p}} z^{\frac{i+1}{2p}}} - \frac{x^{\frac{p-i -1}{2p}} z^{\frac{i+1}{2p}}}{x^{\frac{p-i}{2p}} z^{\frac{i}{2p}}} \right\vert
     = \left\vert \frac{x^{\frac{1}{2p} } }{z^{\frac{1}{2p} } } - \frac{z^{\frac{1}{2p} } }{x^{\frac{1}{2p} } }  \right\vert  = d_s \left( x^{\frac{1}{p}} , z^{\frac{1}{p}}   \right),
   \end{align*}
   and the same holds for $d_s(x, y_{1}) $ and $d_s(y_{p-1}, z) $.

   The last inequality is just a consequence of the concavity of $t\to t^{1/p}$:
   \begin{align*}
     d_s \left( x^{\frac{1}{p}} , z^{\frac{1}{p}}   \right)
     =  \frac{z^{\frac{1}{p}} - x^{\frac{1}{p}}}{(xz)^{\frac{1}{2p}} }  
     = \frac{\frac{1}{p}\int_0^{z-x} (t+x)^{\frac{1-p}{p}} dt}{(xz)^{\frac{1}{2p}} } 
     \leq \frac{\frac{1}{p}\int_0^{z-x} t^{\frac{1-p}{p}} dt}{(xz)^{\frac{1}{2p}} } 
      = \left( \frac{z - x}{(xz)^{\frac{1}{2}} }  \right)^{\frac{1}{p}} \! \! \! \ = d_s(x,z)^{\frac{1}{p}}.
   \end{align*}
  \end{proof} 

\begin{lemma}\label{lem:borne_suite_de_cauchy_D_n}
  Any Cauchy sequence of $(\mathcal D_n^+,d_s)$ is bounded from below and above (in $\mathcal D_n^+$).
\end{lemma}
\begin{proof}
  Considering a Cauchy sequence of diagonal matrices $\Delta^{(k)} \in \mathcal D_n^+$, we know that there exists $K\in \mathbb N$ such that:
  \begin{align*}
    \forall p,q\geq K, \ \forall i\in\{1,\ldots,n\}: \ \ \vert \Delta^{(p)}_i-\Delta^{(q)}_i\vert \leq \sqrt {\Delta^{(p)}_i\Delta^{(q)}_i}.
  \end{align*} 
  For $k\in \mathbb N$, let us introduce the indexes $i^k_M, i^k_m \in \mathbb N$, satisfying:
  \begin{align*}
    \Delta^{(k)}_{i^k_M} = \max \left(\Delta^{(k)}_{i}, 1\leq i \leq n\right)&
    &\text{and}&
    &\Delta^{(k)}_{i^k_M} = \min \left(\Delta^{(k)}_{i}, 1\leq i \leq n\right).
  \end{align*}
  If we suppose that there exists a subsequence $(\Delta_{i_M^k}^{(\phi(k))})_{k\geq 0}$ such that $\Delta_{i_M^k}^{(\phi(k))} \underset{k \to\infty}{\longrightarrow}  \infty$, then
  \begin{align*}
    \sqrt{\Delta^{(\phi(k))}_{i^{\phi(k)}_M}} \leq \sqrt{\Delta^{(N)}_{i^{\phi(k)}_M} } + \frac{\Delta^{(N)}_{i^{\phi(k)}_M}}{\sqrt{\Delta^{(\phi(k))}_{i^{\phi(k)}_M}}}\underset{k \to\infty}{\longrightarrow} \sqrt{\Delta^{(N)}_{i^{\phi(k)}_M} } < \infty
  \end{align*}
  which is absurd. Therefore $(\Delta_{i^{k}_M}^{(k)})_{k\geq 0}$ and thus also $(\Delta^{(k)})_{k\geq 0}$ are bounded from above. For the lower bound, we consider the same way a subsequence $(\Delta_{i_m^k}^{(\psi(k))})_{k\geq 0}$ such that $\Delta_{i_m^k}^{(\psi(k))} \underset{k \to\infty}{\longrightarrow}  0$. We have:
  \begin{align*}
    \Delta^{(\phi(k))}_{i^{\phi(k)}_M} \geq \Delta^{(N)}_{i^{\phi(k)}_M}  - \sqrt{\Delta^{(N)}_{i^{\phi(k)}_M}\Delta^{(\phi(k))}_{i^{\phi(k)}_M}}\underset{k \to\infty}{\longrightarrow} \sqrt{\Delta^{(N)}_{i^{\phi(k)}_M} } >0
  \end{align*}
  which is once again absurd.
\end{proof}
\begin{property}\label{prot:D_n_complet}
  The semi-metric space $(\mathcal D_n^+,d_s)$ is complete.
\end{property}
\begin{proof}
  Given a Cauchy sequence of diagonal matrices $\Delta^{(k)} \in \mathcal D_n^+$, we know from the preceding lemma that there exists $\delta_M,\delta_m\in \mathbb R^+ $ such that $\forall k\geq 0:\delta_m I_n \leq \Delta^{(k)} \leq \delta_M I_n$.
  Thanks to the Cauchy hypothesis:
  \begin{align*}
    \forall \varepsilon >0, \exists K \geq 0 \ | \ \forall p,q \geq K: \ \forall i \in \{1,\ldots,n\}: \left\vert \Delta_i^{(p)} - \Delta_i^{(q)}\right\vert \leq \varepsilon \delta_M
  \end{align*}
  and, as a consequence, $(\Delta^{(k)})_{k\geq 0}$ is a Cauchy sequence in the complete space $(\mathcal D_n^{0,+}, \Vert \cdot \Vert)$: it converges to a matrix $\Delta^{(\infty)} \in \mathcal D_n^{0,+}$. Moreover, $\Delta^{(\infty)}\geq \delta_k I_n$ (as any $\Delta^{(k)}$) for all $k\in \mathbb N$, so that $\Delta^{(\infty)} \in \mathcal D_n^+$ and we are left to showing that $\Delta^{(k)}\underset{k \to\infty}{\longrightarrow}\Delta^{(\infty)}$ for the semi-metric $d_s$. It suffices to write:
  \begin{align*}
    d_s(D^{(k)},D^{(\infty)}) = \left\Vert \frac{D^{(k)}-D^{(\infty)}}{\sqrt {D^{(k)}D^{(\infty)}}}\right\Vert \leq \delta_m\left\Vert D^{(k)}-D^{(\infty)}\right\Vert \underset{k \to\infty}{\longrightarrow} 0.
  \end{align*}
\end{proof}
\begin{proof}[Proof of Theorem~\ref{the:point_fixe_fonction_stable_monotone_de_D_n}]
  We first suppose that $f$ is non-decreasing. As before, let us consider $\delta_M,\delta_m \in \mathbb R^+$ such that $\forall \Delta \in \mathcal D_n^+$ $\delta_m I_n \leq f(\Delta) \leq \delta_M I_n$. The sequence $(\Delta^{(k)})_{k\geq 0}$ satisfying $\Delta^{(0)} = \Delta_m I_n$, and for all $k\geq 1$, $\Delta^{(k)} = f(\Delta^{(k-1)})$ is a non-decreasing sequence bounded superiorly with $\delta_M$, thus it converges to $\Delta^* \in \mathcal D_n^+$ and $\Delta^* = f(\Delta^*)$. This fixed point is clearly unique thanks to \eqref{eq:faiblement_contractant}. 

  Now if $f$ is non-increasing then $\Delta \mapsto f^2(\Delta)$ is non-decreasing and bounded inferiorly and superiorly thus it admits a unique fixed point $\Delta^* \in \mathcal D_n^+$ satisfying $\Delta^* = f^2(\Delta^*)$. We can deduce that $f(\Delta^*) = f^2(f(\Delta^*))$ which implies by uniqueness of the fixed point that $f(\Delta^*) = \Delta^*$ and the uniqueness of such a $\Delta^*$ is again a consequence of \eqref{eq:faiblement_contractant}.
\end{proof}

\end{document}